\documentclass[12pt]{amsart}
\usepackage{enumerate}
\usepackage{tikz}
\usetikzlibrary{decorations.markings}
\usetikzlibrary{shapes}
\usetikzlibrary{backgrounds}
\usepackage{subfig}
\usepackage{graphicx}
\usepackage{amssymb}

\makeatletter
\let\@@enum@org\@@enum@
\def\@@enum@[#1]{\@@enum@org[\normalfont #1]}
\makeatother
\usepackage{hyperref}
\hypersetup{
    pdftitle=   {Hyperbolic surface subgroups of one-ended doubles of free groups},
   pdfauthor=  {Sang-hyun Kim and Sang-il Oum}
}
\newtheorem{THM}{Theorem}
\newtheorem{LEM}[THM]{Lemma}
\newtheorem{COR}[THM]{Corollary}
\newtheorem{PROP}[THM]{Proposition}
\newtheorem{CON}[THM]{Conjecture}
\newtheorem{QUE}{Question}

\theoremstyle{remark}
\newtheorem{EX}[THM]{Example}
\newtheorem*{REM}{Remark}

\theoremstyle{definition}
\newtheorem{DEFN}[THM]{Definition}

\newcommand\form[1]{\langle #1\rangle}
\newcommand\fform[1]{\langle\!\langle #1\rangle\!\rangle}
\newcommand\co{\colon}

\newcommand\Z{\mathbb{Z}}
\newcommand\cay{\operatorname{Cay}}
\newcommand\aut{\operatorname{Aut}}

\newcommand\link{\operatorname{Link}}

\newcommand\abs[1]{\lvert #1\rvert}
\newcommand\osim{\mathord\sim}

\begin{document}
\title[One-ended doubles of free groups]{Hyperbolic surface subgroups of one-ended doubles of free groups}
\author{Sang-hyun Kim}
\address{Department of Mathematical Sciences, KAIST, 291 Daehak-ro, Yuseong-gu, Daejeon 305-701, Republic of Korea}
\email{shkim@kaist.edu}
\thanks{The first named author is supported by the Basic Science Research Program (2011-0026138) and the Mid-Career Researcher Program (2011-0027600) through the National Research Foundation funded by Korea Government.}

\author{Sang-il Oum}
\email{sangil@kaist.edu}
\thanks{The second named author is supported by Basic Science Research
  Program through the National Research Foundation of Korea (NRF)
  funded by the Ministry of Science, ICT \& Future Planning (2011-0011653).}
\date{\today}
\begin{abstract}
Gromov asked whether every one-ended word-hyperbolic group contains a hyperbolic surface group.
We prove that every one-ended double of a free group has a
hyperbolic surface subgroup if (1) the free group has rank two, or 
(2) every generator is used the same number of times 
in a minimal automorphic image of the
amalgamating words.
To prove this, we formulate a stronger statement on Whitehead
graphs and prove its specialization by 
combinatorial induction for (1) 
and the characterization of perfect matching polytopes by Edmonds for (2).

\end{abstract}
\maketitle

\section{Introduction}
A \emph{hyperbolic surface group} is the fundamental group of a closed surface with negative Euler characteristic. 
We will denote by $F_n$ the free group of rank $n$ with a fixed basis $\mathcal{A}_n=\{a_1,\ldots,a_n\}$.
A \emph{double of a free group} 
is the fundamental group of a graph of groups where there are two free vertex groups and 
at least one infinite cyclic edge group; here, each edge group is amalgamated along the copies of some word in the free group
(Figure~\ref{fig:double}). 
If $U$ is a list of words in $F_n$,
we denote by $D(U)$ the double of $F_n$ where a cyclic edge group is glued along the copies of each word in $U$.

We study the existence of hyperbolic surface subgroups
in doubles of free groups. This is motivated by the following remarkable question due to Gromov.

\begin{QUE}[{Gromov~\cite[p. 277]{Gromov1993}}]\label{que:gromov}
	Does every one-ended word-hyperbolic group have a hyperbolic surface subgroup?
\end{QUE}

Question~\ref{que:gromov} has been answered affirmatively for the following cases.
\begin{enumerate}
	\item
	Coxeter groups~\cite{GLR2004}.
	\item\label{result:calegari}	Graphs of free groups with infinite cyclic edge groups with nontrivial second rational homology~\cite{Calegari2009}.
	\item
	The fundamental groups of closed hyperbolic $3$-manifolds~\cite{KM2012}.
\end{enumerate}

A basic, but still captivating case is when the group is given as a double of a free group.
Using (2), Gordon and Wilton~\cite{GW2010} constructed explicit families of examples of 
doubles that contain hyperbolic surface groups; they showed that those families virtually have nontrivial second rational homology.
The existence of a hyperbolic surface subgroup is not known for
doubles with trivial virtual second rational homology. This leads us
to the next question.

\begin{QUE}\label{que:double}
Does every one-ended double of a nonabelian free group have a hyperbolic surface subgroup?
\end{QUE}

Our first main result resolves Question~\ref{que:double} for rank-two case:

\begin{THM}\label{thm:main 1}
A double of a rank-two free group is one-ended if and only if it has a hyperbolic surface subgroup.
\end{THM}

Let $U$ be a list of words in $F_n$.
When approaching Question~\ref{que:double} for $D(U)$,
we may always assume that  $U$ is \emph{minimal} in the sense that no automorphism of $F_n$ reduces the sum of the lengths of the words in $U$. This is because the isomorphism type of $D(U)$ is invariant under the automorphisms of $F_n$.
By the same reason, we may assume that each word in $U$ is cyclically reduced.
We say $U$ is \emph{$k$-regular} if each generator in $\mathcal{A}_n$ appears exactly $k$ times in $U$. Our second main result answers Question~\ref{que:double} affirmatively for a minimal, $k$-regular list of words.

\begin{THM}\label{thm:main 2}
Suppose $U$ is a minimal, $k$-regular list of cyclically reduced words in $F_n$ where $k,n>1$.
If $D(U)$ is one-ended, then $D(U)$ contains a hyperbolic surface group.
\end{THM}

Here is an overview of our proof.  
We first explain why Tiling Conjecture~\cite{KW2012,Kim2009} implies
an affirmative answer for Question~\ref{que:double} in Section~\ref{sec:polygonal}.
And then we reformulate Tiling Conjecture into a purely graph theoretic conjecture in Section~\ref{sec:whitehead}.
We resolve this graph theoretic conjecture in two special cases. 
In Section~\ref{sec:regular}, we prove it for regular graphs and
deduce Theorem~\ref{thm:main 2}. Here we use the characterization of
perfect matching polytopes of graphs by Edmonds~\cite{Edmonds1965a}.
In Section~\ref{sec:4vertex}, we prove it for 4-vertex graphs and
deduce Theorem~\ref{thm:main 1}. 

\subsection*{Acknowledgement.} 
Authors would like to thank Daniel Kr\'al' for pointing out Example~\ref{ex:connectivity}.
They thank the anonymous referee for numerous helpful suggestions.

\section{Polygonality and Tiling Conjecture}\label{sec:polygonal}
In this section, we give a review on polygonality and its implication.

\subsection{Doubles and polygonality}\label{subsec:polygonality}
Let $n\ge1$. Each word in $F_n$ can be written as $w=x_1 x_2\cdots x_\ell$ where $x_i\in \mathcal{A}_n\cup \mathcal{A}_n^{-1}$. We will often take the subscript of $x_i$ modulo $\ell$. 
If $w$ is reduced, we call $\ell$ as the \emph{length} of $w$ and write $\abs{w} = \ell$.
We denote the Cayley graph of $F_n$ by $\cay(F_n)$.
There is a natural free action of $F_n$ on $\cay(F_n)$ so that $\cay(F_n)/F_n$ is a bouquet of circles. 
Let $\alpha_1,\ldots,\alpha_n$ denote the oriented circles in $\cay(F_n)/F_n$ corresponding to $a_1,\ldots,a_n\subseteq \mathcal{A}_n$.
The loop obtained by a concatenation 
$\alpha_i^p\alpha_j^q\cdots \alpha_k^r$ where $p,q,\ldots,r\in \Z$ is said to 
\emph{read} the word $a_i^p a_j^q\cdots a_k^r$.

Let $U=(u_1,u_2,\ldots,u_r)$ be a list of nontrivial words in $F_n$.
Take two copies $\Gamma$ and $\Gamma'$ of $\cay(F_n)/F_n$.
To $\Gamma$ and $\Gamma'$, we glue a cylinder along the copies of the closed curve reading $u_i$, for each $i$.
Let $X(U)$ be the resulting space and let $D(U)=\pi_1(X(U))$ be the fundamental group of $X(U)$; see Figure~\ref{fig:double}.
In the literature, $D(U)$ is called a \emph{double of $F_n$ along $U$}, or simply a \emph{double}~\cite{BFMT2009}.
If we let $\mathcal{B}_n$ and $V=\{v_1,\ldots,v_r\}$ denote the copies of $\mathcal{A}_n$ and $U$ respectively, then a presentation of $D(U)$ is given as:
\[D(U) \cong \form{\mathcal{A}_n, \mathcal{B}_n,t_2,t_3,\ldots,t_r 
\mid u_1=v_1, u_i^{t_i}=v_i\mbox{ for }i=2,\ldots,r}.\]
Since the isomorphism type of $D(U)$ does not change if some words in $U$ are replaced by their conjugates, we will always assume that every word in $U$ is cyclically reduced. Note that $U$ is possibly redundant.

\begin{figure}[tb]
\includegraphics[width=.6\textheight]{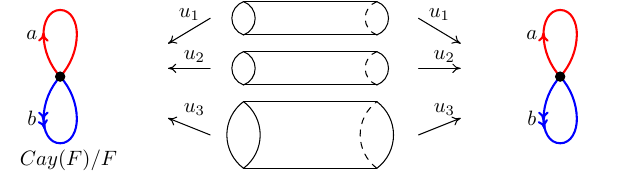}
\caption{A construction of $X(U)$, where $\pi_1(X(U))=D(U)$.
}
\label{fig:double}\end{figure}

For a word $w = x_1x_2\cdots x_\ell\in F_n$, 
the  \emph{length-two cyclic subwords} of $w$
will refer to the following words:
\[x_1x_2,x_2x_3,\ldots,x_{\ell-1}x_\ell,x_\ell x_1.\]

The \emph{Whitehead graph $W(U)$ for $U$} is constructed as follows~\cite{Whitehead1936}:
\begin{enumerate}[(i)]
\item
the vertex set of $W(U)$ is $\mathcal{A}_n\cup\mathcal{A}_n^{-1}$;
\item
for each length-two cyclic subword $xy$ of a word in $U$, 
we add an edge joining $x$ and $y^{-1}$ to $W(U)$.
\end{enumerate}
We will later give an equivalent, alternative formulation of a Whitehead graph as an abstract graph with additional structures; see Lemma~\ref{lem:whitehead}.

A \emph{polygonal disk} means a topological $2$-disk $P$ equipped with a graph structure on the boundary $\partial P\approx S^1$.
We let $Z(U)$ denote the presentation $2$-complex of $F_n/\fform{U}$. This means that $Z(U)$ is obtained from its 1-skeleton $\cay(F_n)/F_n$ by attaching a polygonal disk $D_i$ along the loop reading $u_i$ for each $i=1,2,\ldots,r$. Here, $\partial D_i$ is regarded as a $\abs{u_i}$-gon.
As before, we let $\alpha_j$ denote the oriented loop in $Z(U)^{(1)}=\cay(F_n)/F_n$ reading $a_j$.
The link of the unique vertex in $Z(U)$ is the Whitehead graph for $U$ after identifying the incoming (outgoing, respectively) portion of $\alpha_j$ with the vertex $a_j$ ($a_j^{-1}$, respectively) in $W(U)$.

Let us fix a point $d_i$ in the interior of $D_i$ and triangulate $D_i$ so that each triangle contains $d_i$ and one edge of $\partial D_i$.
Remove a small open neighborhood of $d_i$ for each $i$ to get a square complex $Z'$; see Figure~\ref{fig:zprime}(a). 
We obtain a square complex structure on $X(U)$ by taking two copies of $Z'$ and gluing the circles corresponding to the boundary of the neighborhood of each $d_i$.
The unique vertex of $Z(U)$ gives two special vertices of $X(U)$.
Note that the link of each special vertex is the barycentric subdivision $W(U)'$ of $W(U)$.
Since $W(U)$ has no loops, we see $W(U)'$ is a bipartite graph without parallel edges.
We remark that $X(U)$ is a non-positively curved square complex by Gromov's link condition~\cite{Gromov1987}.

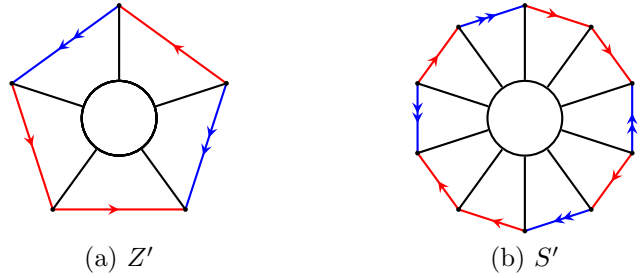
\begin{figure}[t]
  \tikzstyle {a}=[red,postaction=decorate,decoration={%
    markings,%
    mark=at position .5 with {\arrow[red]{stealth};}}]
  \tikzstyle {b}=[blue,postaction=decorate,decoration={%
    markings,%
    mark=at position .43 with {\arrow[blue]{stealth};},%
    mark=at position .57 with {\arrow[blue]{stealth};}}]
  \tikzstyle {v}=[draw,circle,fill=black,inner sep=0pt]
 \subfloat[(a) $Z'$]{\begin{tikzpicture}[thick]
  \foreach \i in {1,2,4}
    \draw [a] (360/5*\i+90:1.5)--(360/5*\i+360/5+90:1.5);
   \draw [b] (90:1.5)--(90+360/5:1.5);
   \draw [b] (360/5*4+90:1.5)--(90+360/5*3:1.5);
   \foreach \i in {0,...,4} {
   \node [circle,draw,inner sep=10pt] at (0,0) (c) {};
  \draw (c)-- (360/5*\i+90:1.5) node [v] {} ;
  }
 \node  [inner sep=0.9pt] at (-90:1.5) {};  
 \end{tikzpicture}}\quad\quad\quad\quad\quad\quad
\subfloat[(b) $S'$]{\begin{tikzpicture}[thick]
   \foreach \i in {0,2,4,5,7,9}
    \draw [a] (360/10*\i+90:1.5)--(360/10*\i-360/10+90:1.5);
    \foreach \i in {1,6}
    \draw [b] (360/10*\i+90:1.5)--(360/10*\i-360/10+90:1.5);
    \foreach \i in {2,7}
    \draw [b] (360/10*\i+90:1.5)--(360/10*\i+360/10+90:1.5);
   \node [circle,draw,inner sep=10pt] at (0,0) (c) {};
  \foreach \i in {0,...,9} {
  \draw (c)-- (360/10*\i+90:1.5) node [v] {} ;
   }
\end{tikzpicture}}%
\caption{Square complex structures on $Z'$ and on $S'$. 
A single and a double arrow denote the generators $a$ and $b$, respectively.
Figure (a) shows a punctured $D_i$ in $Z'$, divided into squares.
Figure (b) is a punctured $P_i$ in $S'$, where $\partial P_i\to\cay(F)/F$ reads $(b^{-1}aba^2)^2$.
}
\label{fig:zprime}
\end{figure}

A \emph{side-pairing} on polygonal disks $P_1,\ldots,P_m$ is an
equivalence relation on the sides of $P_1,\ldots,P_m$ such that each equivalence class consists of two sides, along with a choice of a homeomorphism between the two sides of each equivalence class. 
For a given side-pairing $\sim$ on polygonal disks $P_1,\ldots,P_m$, one gets a closed surface $S=\coprod_i P_i/\osim$ by identifying the sides of $P_i$ by $\sim$. The surface $S$ is naturally equipped with a two-dimensional CW-structure.
A graph map $\phi\co G\to\cay(F_n)/F_n$ induces an orientation and a label by $\mathcal{A}_n$ on each edge $e$ of $G$, so that the oriented loop $\phi(e)$ reads the label of $e$.
An edge labeled by $a_i$ is called an \emph{$a_i$-edge}.
An \emph{immersion} will mean a locally injective graph map.

\begin{DEFN}[\cite{KW2012,Kim2009}]\label{defn:polygonal}
A list $U$ of cyclically reduced words in $F_n$ is called \emph{polygonal} if there exist a side-pairing $\sim$ on some polygonal disks $P_1,P_2,\ldots,P_m$
and an immersion $S^{(1)}\to \cay(F_n)/F_n$ where $S=\coprod_i P_i/\osim$ such that the following hold:
\begin{enumerate}[(i)]
	\item the composition $\partial P_i\to S^{(1)}\to\cay(F_n)/F_n$ reads a nontrivial power of a word in $U$ for each $i$;
	\item the Euler characteristic $\chi(S)$ of $S$ is less than $m$.
\end{enumerate}
In this case, we call $S$ a \emph{$U$-polygonal surface}.
\end{DEFN}

\begin{REM}\label{rem:independence}
Polygonality of a list of words depends on the choice of a free-basis. An example given in~\cite{KW2012} is the word $w=abab^2ab^3$ in $F_2=\form{a,b}$. It was shown that while $w$ is not polygonal, the automorphism $(a\mapsto ab^{-2},b\mapsto b)$ maps $w$ to a polygonal word $ab^{-1}a^2b$.
\end{REM}

Polygonality was originally defined 
for a word~\cite{KW2012} and then for a set of words~\cite{Kim2009}.
Here, we generalize to a (possibly redundant) list of words. 
It is immediate that the main implication of polygonality still holds, as described below.

\begin{THM}[\cite{KW2012,Kim2009}]\label{thm:polygonal}
If $U$ is a polygonal list of words in $F_n$, then $D(U)$ contains a hyperbolic surface group.
\end{THM}

We remark that similar (but not identical) ideas of exploring surface subgroups in a given group by gluing polygonal disks can also be found in~\cite{Schupp1980,Culler1981,Calegari2009,BCF2011}.

\subsection{Tiling Conjecture and its implication}\label{subsec:split}

 A list  $U$ of words in $F_n$ is said to be \emph{diskbusting} if one cannot write $F_n=A\ast B$ in such a way that $A,B\ne\{1\}$ and each word in $U$ is conjugate into $A$ or $B$~\cite{Canary1993,Stong1997,Stallings1999}. We note that $D(U)$ is one-ended if and only if $U$ is diskbusting~\cite{GW2010}.

\begin{CON}[Tiling Conjecture; see \cite{KW2012,Kim2009}]\label{con:tiling}
 A minimal and diskbusting list of cyclically reduced words in $F_n$ is polygonal when $n>1$.
\end{CON}

If Tiling Conjecture is true, then one would be able to precisely describe when doubles contain hyperbolic surface groups as follows; in particular, Question~\ref{que:double} would have an affirmative answer. 

\begin{PROP}\label{prop:tilingequiv}
Let $n>1$.  
Suppose that every minimal and diskbusting list of cyclically reduced words in $F_m$ is polygonal for all $m=2,3,\ldots,n$.
Then for a list $U$ of cyclically reduced words in $F_n$, the double $D(U)$ contains a hyperbolic surface group if and only if $F_n$ cannot be written as 
$F_n=G_1\ast G_2\ast\cdots G_n$ in such a way that 
each $G_i$ is infinite cyclic and each word in $U$ is conjugate into one of 
$G_1,\ldots,G_n$.
\end{PROP}

\begin{proof}
The forward implication is immediate from that every double of $\Z$ is virtually isomorphic to $\Z\times F_s$ for some $s\ge0$.

For the backward implication, choose the maximum $k$ such that $F_n=G_1\ast\cdots\ast G_k$ for some nontrivial groups $G_1,\ldots,G_k$ and each word in $U$ is conjugate into one of the $G_1,\ldots,G_k$. Assume $k < n$.
We may assume that $G_1$ has rank $m > 1$.
Let $U_1$ be the list of all the words in $U$ conjugate into $G_1$.
Then suitably chosen conjugates of the words in $U_1$ form a diskbusting list $U_1'$ in the rank-$m$ free group $G_1$.
We note that $D(U_1') \le D(U_1'\cup (U\setminus U_1)) \cong D(U)$.
From the hypothesis, a free basis $\mathcal{B}$ of $G_1$ can be chosen so that  $U_1'$ is polygonal as a list of words written in $\mathcal{B}$.
By Theorem~\ref{thm:polygonal}, $D(U_1')$ contains a hyperbolic surface group.
\end{proof}

\section{Combinatorial Formulation of Tiling Conjecture}\label{sec:whitehead}
\subsection{Terminology on graphs}
We allow graphs to have parallel edges or loops;
a \emph{loop} is an edge with only one endpoint.
For a graph $G$, we write $V(G)$ and $E(G)$ to denote the vertex set and the edge set of $G$,
respectively. The \emph{degree} $\deg_G(v)$ of a vertex $v$ is the number of
edges incident with $v$, assuming that loops are counted twice.
For a set $X$ of vertices, we write $\delta_G(X)$ to denote the set of edges having endpoints in both $X$ and $V(G)\setminus X$.
In particular, the set $\delta_G(v)$ consists of non-loop edges incident with $v$. 
A graph is \emph{$k$-regular} if every vertex has degree $k$,
and it is \emph{regular} if it is $k$-regular for some $k\ge0$.
A \emph{cycle} is a (finite) $2$-regular connected graph. 
For two distinct vertices $x$ and $y$ of a graph $G$, 
the \emph{local edge-connectivity} $\lambda_G(x,y)$ is the maximum number of
pairwise edge-disjoint paths from $x$ to $y$ in $G$. We omit the
subscript $G$ in $\deg_G$, $\delta_G$ and $\lambda_G$ if the underlying graph $G$ is clear from the context. 
Menger's theorem~ states that $\lambda(x,y)=\min\{ \abs{\delta(X)} \co
x\in X, y\not\in X\}$; see~\cite[Chapter 3]{Diestel_book}
or~\cite{Menger1927}.

A \emph{pairing} of a graph $G=(V,E)$ is a fixed-point-free involution on $V$.
A graph $G$ with a pairing $v\mapsto v^{-1}$ is called \emph{pairwise well-connected}
if $\lambda(v,v^{-1})=\deg(v)$ for each vertex $v$ of $G$.
Whitehead graphs are always equipped with a canonical pairing
$v\mapsto v^{-1}$.

\subsection{Whitehead graph and associated connecting maps}\label{subsec:connecting}
Let $U$ be a list of cyclically reduced words in $F_n$.
The following characterization of a minimal set of words is given
in~\cite[Section 8]{Berge1990}: the list $U$ is not minimal if and
only if for some $i$, there exists a set $C$ of edges in the Whitehead
graph $W(U)$ such that $\abs{C}<\deg(a_i)$ and $W(U)\setminus
C$ has no path from $a_i$ to $a_i^{-1}$.
If there is $C\subseteq E(W(U))$ such that $W(U)\setminus C$ has no path
from $a_i$ to $a_i^{-1}$, 
then there is a set $X\subseteq V(W(U))$ such that $\delta(X)\subseteq C$
and 
$a_i\in X$, $a_i^{-1}\notin X$.
By Menger's theorem,
it follows that $U\subseteq F_n$ is minimal if and only if
$W(U)$ is pairwise well-connected.
A minimal set $U\subseteq F_n$ is diskbusting if and only if $W(U)$ is connected~\cite{Whitehead1936,Stong1997,Stallings1999}.
These results on sets of words immediately generalize to lists of words:

\begin{PROP}[\cite{Whitehead1936,Berge1990,Stong1997,Stallings1999}]\label{prop:whitehead}
The list $U$ is minimal and diskbusting if and only if $W(U)$ is connected and 
pairwise well-connected.
\end{PROP}

The Whitehead graph $W(U)$ has an \emph{associated involution} on $\mathcal A_n\cup \mathcal A_n^{-1}$ defined by $v\mapsto v^{-1}$ for all $v\in A_n\cup \mathcal A_n^{-1}$.
Let  $w=x_1x_2\cdots x_\ell$ be a word in $U$,
and $p$ and $q$ be the edges of $W(U)$ corresponding to $x_{i-1}x_{i}$ and $x_{i}x_{i+1}$, respectively.
Then we write $p^{x_i} = q$ and $q^{x_i^{-1}}=p$. Note that for each vertex $v$, the map $e\mapsto e^v$ defines a bijection from $\delta(v^{-1})$ to $\delta(v)$. This bijection is called the \emph{associated connecting map at $v$}.
It is immediate that $(e^{v})^{v^{-1}} = e$ for each vertex $v$ and each edge $e$ in $\delta(v^{-1})$.
We will always denote the associated involution and the associated connecting maps of a Whitehead graph as $v\mapsto v^{-1}$ and $e\mapsto e^v$.

If $(e^x)^y$ is well-defined for an edge $e$ and vertices $x\ne y^{-1}$ of $W(U)$, then there exists a word $w$ in $U$ and $s,t\in V(W(U))$ such that $e, e^x$ and $(e^x)^y$ correspond to some length-two cyclic subwords of $w$ or $w^{-1}$ as shown below:
\[w\text{ or }w^{-1}=\cdots\lefteqn{\overbrace{\phantom{s\cdot x}}^e}s\cdot\lefteqn{\underbrace{\phantom{x\cdot y}}_{e^x}}x\cdot \overbrace{y\cdot t}^{(e^x)^y}\cdots.\]
The proof of the following observation is then elementary.
\begin{LEM}\label{lem:reading}
In $W(U)$, consider an edge $f$ and vertices $x_1,x_2,\ldots,x_\ell$ where $\ell>0$, such that $x_{i+1}\ne x_i^{-1}$ for $i=1,\ldots,\ell$.
Suppose that 
\[(\cdots((f^{x_1})^{x_2})\cdots)^{x_\ell}\]
is well-defined and equal to $f$.
Then $x_1 x_2 \cdots x_\ell $ is a nontrivial power of a cyclic conjugation of a word in $U$.\qed
\end{LEM}

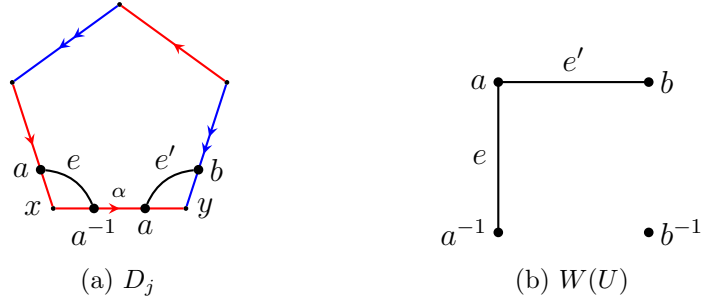
\begin{figure}
  \tikzstyle {a}=[red,postaction=decorate,decoration={%
    markings,%
    mark=at position .5 with {\arrow[red]{stealth};}}]
  \tikzstyle {b}=[blue,postaction=decorate,decoration={%
    markings,%
    mark=at position .43 with {\arrow[blue]{stealth};},%
    mark=at position .57 with {\arrow[blue]{stealth};}}]
  \tikzstyle {v}=[draw,shape=circle,fill=black,inner sep=0pt]
  \tikzstyle {bv}=[black,draw,shape=circle,fill=black,inner sep=1pt]
  \tikzstyle{every edge}=[-,draw]
\subfloat[(a) $D_j$]{
	\begin{tikzpicture}[thick]
   		\foreach \i in {0,...,4} 
			\draw (360/5*\i+90:1.5) node [v] (v\i) {} ;
		\draw[a] (v1)--(v2)
		node[pos=0.7,bv] (ap1) {};
		\draw[a] (v2)--(v3)
		node[pos=.5] (alph) {}
		node[pos=0.3,bv] (am) {}
		node[pos=.7,bv] (ap2) {};
		\draw[a] (v4)--(v0);
		\draw [b] (v0)--(v1);
		\draw [b] (v4)--(v3)
		node[pos=0.7,bv] (bp) {};
		\draw (v2) node [left] {$x$};
		\draw (v3) node [right] {$y$};
		\draw (alph) node [above] {${}_\alpha$};
		\draw [bend left] (ap1)  edge  node [above] {$e$} (am)
		(ap2)  edge  node [above] {$e'$} (bp);
		\draw (am)  node [below] {$a^{-1}$};
		\draw (ap1) node [left] {$a$};	
		\draw (ap2) node [below] {$a$};
		\draw (bp) node [right] {$b$};	
	\end{tikzpicture}
}
\quad\quad\quad\quad\quad\quad
  \subfloat[(b) $W(U)$]{\begin{tikzpicture}[thick]
      \draw  (-1,-1) node  [bv]  (am) {} node[left] {$a^{-1}$} 
      -- node[midway,left] {$e$} (-1,1) node [bv] (ap) {} node[left] {$a$} 
     -- node[midway,above]{$e'$} (1,1) node [bv] (bp) {} node[right]{$b$}
      (1,-1) node [bv] {} node[right]{$b^{-1}$};
\end{tikzpicture}}%
\caption{Each corner of a cell $D_j$ in $Z(U)$ corresponds an edge in $W(U)$.
Here, $F_2=\form{a,b}$ and $U=\{b^{-1}aba^2\}$. 
In these two figures, we note that $e^a=e'$.
}
\label{fig:connecting}
\end{figure}

Connecting maps can also be described in $Z(U)$. 
The link of a vertex $p$ in a polygonal disk $P$ is called the \emph{corner} of $P$ at $p$.
Suppose an edge $e$ is incident with $a_i^{-1}$ in $W(U)$, where $e$ corresponds to the corner of a vertex $x$ in some $D_j$ attached to $Z(U)$.
Since we are assuming that every word in $U$ is cyclically reduced, there exists a unique $a_i$-edge $\alpha$ outgoing from $x$. 
Choose the other endpoint $y$ of $\alpha$, and let $e'\in E(W(U))$ correspond to the corner of $D_j$ at $y$; see Figure~\ref{fig:connecting}.
Then we observe that $e^{a_i}=e'$.

\subsection{Graph-theoretic formulation of Tiling Conjecture}
Polygonality was described in terms of Whitehead graphs~\cite[Propositions~17 and~21]{Kim2009}.
But this description required infinitely many graphs to be examined.
In this subsection, we obtain a simpler formulation of polygonality requiring only one finite graph to be examined.

The following is immediate from the definition.
\begin{LEM}\label{lem:whitehead}
Let $G= (V,E)$ be a graph with a pairing $v\mapsto v^{-1}$.
Suppose that for each $v$ there is a bijection from $\delta(v^{-1})$ to $\delta(v)$
denoted as $e\mapsto e^v$.
If \[(e^v)^{v^{-1}}=e\] for each $v\in V$ and $e\in \delta(v^{-1})$,
then there exist a list $U$ of words $F_{\abs{V}/2}$ and a graph
isomorphism $\phi\co G\to W(U)$ such that the paring and the bijections correspond to the associated involution and to the associated connecting maps by $\phi$, respectively.
\end{LEM}

From now on, we will simply write ``$G$ is a Whitehead graph''
when we have a graph $G$ with a pairing $v\mapsto v^{-1}$ and bijections
$e\in \delta(v^{-1})\mapsto e^v\in \delta(v)$ for each vertex $v$ 
such that the hypothesis of Lemma~\ref{lem:whitehead} is satisfied.


A non-empty finite list of cycles of a Whitehead graph $G$ is called
\emph{balanced}
if 
\begin{enumerate}[(i)]
\item
$\mathcal{C}$ has at least one cycle of length at least three,
\item
for each pair of edges $e$ and $f$ incident with a vertex $v^{-1}$, the number of cycles in $\mathcal{C}$ containing both $e$ and $f$ is equal to the number of cycles in $\mathcal{C}$ containing both $e^v$ and $f^v$.
\end{enumerate}


\begin{LEM}\label{lem:equiv}
Let $n>1$. A list  $U$ of cyclically reduced words in $F_n$ is polygonal if and only if $W(U)$ admits a balanced list of cycles.
\end{LEM}
By Lemma~\ref{lem:equiv}, we can restate Tiling Conjecture combinatorially:

\begin{CON}\label{con}
A connected and pairwise well-connected Whitehead graph on at least
$4$ vertices
admits a balanced list of cycles.
\end{CON}

Now we provide the proof of Lemma~\ref{lem:equiv}.
We prove the forward direction by similar arguments to~\cite[Propositions~17 and 21]{Kim2009}.
The backward direction is what we mainly need for this paper.

\begin{proof}[Proof of Lemma~\ref{lem:equiv}]
To prove the forward direction, let us assume $U$ is polygonal; 
we can find a $U$-polygonal surface $S=\coprod_{1\le i\le m} P_i/\osim$ as in Definition~\ref{defn:polygonal}. 
In particular, each edge in $S^{(1)}$ is oriented and labeled by $\mathcal{A}_n$. Put $S^{(0)}=\{v_1,\ldots,v_t\}$. 
Fix $p_i$ in the interior of each $P_i$. 
There exists a natural map $\phi \co S\to Z(U)$ such that $\phi $ is locally injective away from $p_1,\ldots,p_m$. 
Since $S$ is a closed surface and 
$\phi$ is locally injective at $v_i$,
the image of each $\link_S(v_i)$ 
by $\phi $ is a cycle, say $C_i$, in $W(U)$.

Choose a vertex $v\in W(U)$ and two edges $e,f$ incident with $v$. 
Without loss of generality, we may assume that $v=a^{-1}$ for some generator $a\in\mathcal{A}_n$ and $C_1,\ldots, C_{t'}$ is the list of the cycles among $C_1,\ldots,C_t$ which contain both $e$ and $f$. 
For each $i=1,\ldots,t'$, there exists a unique $a$-edge $e_i$ outgoing from $v_i$. Let $v_{\sigma(i)}$ be the endpoint of $e_i$ other than $v_i$. 
There exist exactly two polygonal disks $Q_i$ and $R_i$ sharing $e_i$ in $S$, so that $\link(\phi;v_i)\co\link(v_i)\to W(U)$ sends the corner of $Q_i$ at $v_i$ to $e$, and that of $R_i$ at $v_i$ to $f$. By the definition of a connecting map, $\link(\phi;v_{\sigma(i)})\co\link(\sigma(v_i))\to W(U)$ maps the corners of $Q_i$ and $R_i$ at $v_{\sigma(i)}$ to $e^a$ and $f^a$, respectively; 
see Figure~\ref{fig:equiv}.
The correspondence $e\cup f \to e^a\cup f^a$ defines  an involution on the list of length-two subpaths of $ C_1,\ldots, C_t$. The conclusion follows.

\begin{figure}
  \tikzstyle {a}=[postaction=decorate,decoration={%
    markings,%
    mark=at position .5 with {\arrow[]{stealth};}}]
  \tikzstyle {b}=[postaction=decorate,decoration={%
    markings,%
    mark=at position .43 with {\arrow[]{stealth};},%
    mark=at position .57 with {\arrow[]{stealth};}}]
  \tikzstyle {c}=[postaction=decorate,decoration={%
    markings,%
    mark=at position .4 with {\arrow[]{stealth};},%
    mark=at position .5 with {\arrow[]{stealth};},
    mark=at position .6 with {\arrow[]{stealth};}
}]
  \tikzstyle {v}=[draw,shape=circle,fill=black,inner sep=0pt]
  \tikzstyle {bv}=[black,draw,shape=circle,fill=black,inner sep=1pt]
  \tikzstyle{every edge}=[-,draw]
  \subfloat[(a) $S$]{
  \begin{tikzpicture}[thick]
    \draw (-3,0) node[] (v1) {}
	[a]	-- (-1.5,0) node [bv] (v2) {};
	\node[left] at (-1.9,.3) {$v_i$};
	\draw (v2)
	[a]-- 
	node[pos=.2,v] (e2) {}
	node[pos=.5,above,black] {$e_i$}
	node[pos=.8,v] (pe2) {}
 	(1.5,0) node [bv] (v3) {};
	\node[right] at (1.9,.3) {$v_{i'}$};
	\draw (v3)
	[a]-- 
	(3,0) node [] (v4) {};
	\draw (-2.5,1.5) node[] (v5) {}
	[b]--
	node[pos=.2,v] (e1) {}
	(v2);
	\draw [orange,in=120,out=-30] (e1) edge node [above,black] {$e$} (e2);
	\draw (v2) 
	[b]-- 
	node[pos=.8,v] (e3) {}
	(-2.5,-1.5) node[] (v6) {};
	\draw [orange,out=-120,in=30] (e2) edge node [below,black] {$f$} (e3);	
	\draw (v3) 
	[c]--
	node[pos=.8,v] (pe1) {}
	(2.5,1.5) node[] (v7) {};
	\draw [purple,in=60,out=210] (pe1)
	edge
	node [left,pos=.1,black,inner sep=10pt] {$e^a$}
	(pe2);
	\draw (2.5,-1.5) node[] (v8) {}
	[c]-- 
	node[pos=.2,v] (pe3) {}
	(v3);	
	\draw [purple,out=-60,in=150] (pe2)
	edge
	node [left,pos=.9,black,inner sep=10pt] {$f^a$}
	(pe3);
	\draw[left] (0,1.2) node[] {$Q_i$};
	\draw[left] (0,-1.2) node[] {$R_i$};
\node  [inner sep=0.9pt] at (-90:1.5) {};  
   \end{tikzpicture}
}
\quad\quad
\subfloat[(b) $W(U)$]{\begin{tikzpicture}[thick]
    \draw  (-1.5,-1) node  [bv]  (bm) {} node[left] {$b^{-1}$} 
    [orange]-- node[pos=.3,above,black] {$f$} (0,-1) node [bv] (am) {} node [right,black]{$a^{-1}$};
 	\draw (am)
	[orange]-- node[midway,right,black] {$e$} 
	(-1.5,1) node [bv] (bp) {} node[left,black]{$b$};
	\draw
	(1.5,1) node [bv] (cp) {} node [right,black] {$c$} 
	[purple]-- node [midway,above,black] {$f^a$} 
	(0,1) node [bv] (ap) {} node [left,black] {$a$};
	\draw
	(ap)
	[purple]-- node [midway, right,black]  {$e^a$}
	(1.5,-1) node [bv] (cm) {} node [right,black]  {$c^{-1}$};
\end{tikzpicture}}%
\caption{Consecutive corners in $S$ and their images by a connecting map.
$F_3 = \form{a,b,c}$, and single, double and triple arrows denote the labels $a,b$ and $c$, respectively.
}
\label{fig:equiv}
\end{figure}
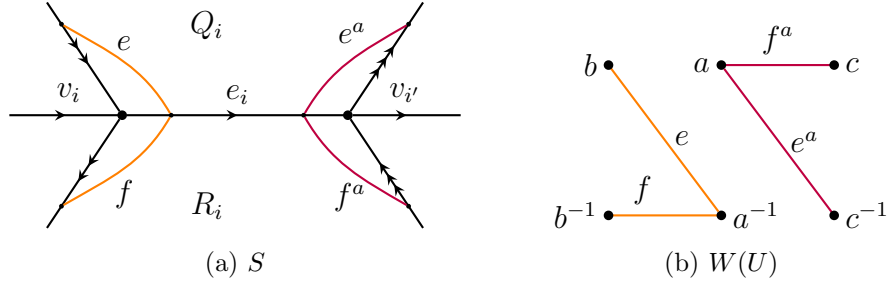

For the backward direction, consider a balanced list of cycles $ C_1,\ldots, C_t$ in $W(U)$.
For each $C_i$, let $V_i$ be a polygonal disk such that 
$\partial V_i$ is a cycle of the same length as $C_i$.
We will regard $\partial V_i$ as the dual cycle of $ C_i$, in the sense that each edge of $\partial V_i$ corresponds to a vertex of $C_i$ and incident edges correspond to adjacent vertices.
Choose a linear order $\prec$ on $\{(v,e):  e\in\delta(v)\}$ for each $v\in V(W(U))$ such that $(v,e)\prec (v,e')$ if and only if $(v^{-1},e^{v^{-1}})\prec (v^{-1}, (e')^{v^{-1}})$.
If an edge $g$ of $\partial V_i$ corresponds to the vertex $v\in V(W(U))$ of $C_i\subseteq W(U)$, then we will label the edge $g$ by $(v,\{e,f\})$ where $e$ and $f$ are the two edges of $C_i$ incident with $v$; see Figure~\ref{fig:equiv2} (a) and (b).
Considered as a side of $V_i$, the edge $g$ will be given with a transverse orientation, 
which is incoming into $V_i$ if $v\in\mathcal{A}_n$ and outgoing if $v\in\mathcal{A}_n^{-1}$.
If $w_e$ and $w_f$ denote the vertices of $g$ corresponding to $e$ and $f$ respectively, 
and $(v,e)\prec (v,f)$, then we shall orient $g$ from $w_f$ to $w_e$.
Define a side-paring $\sim_0$ on $V_1,\ldots,V_t$ such that
$\sim_0$ respects the orientations and moreover,
a side labeled by $(v,\{e,f\})$ is paired with a side labeled by $(v^{-1},\{e^{v^{-1}},f^{v^{-1}}\})$ for each $v\in V(W(U))$ and $e,f\in\delta(v)$ where $e$ and $f$ are consecutive edges of some cycle $C_i$. Such a side-pairing exists since $C_1,\ldots,C_t$ is balanced.
Consider the closed surface $S_0 = \coprod_i V_i/\osim_0$. 
Denote by $\eta$ and $\zeta$ the numbers of the edges and the faces in $S_0$, respectively. 
Each edge in $S_0$ is shared by two faces, and each face has at least two edges; 
moreover, at least one face has more than two edges.
So $2\zeta < \sum_i (\textrm{the number of sides in }V_i) = 2\eta$.

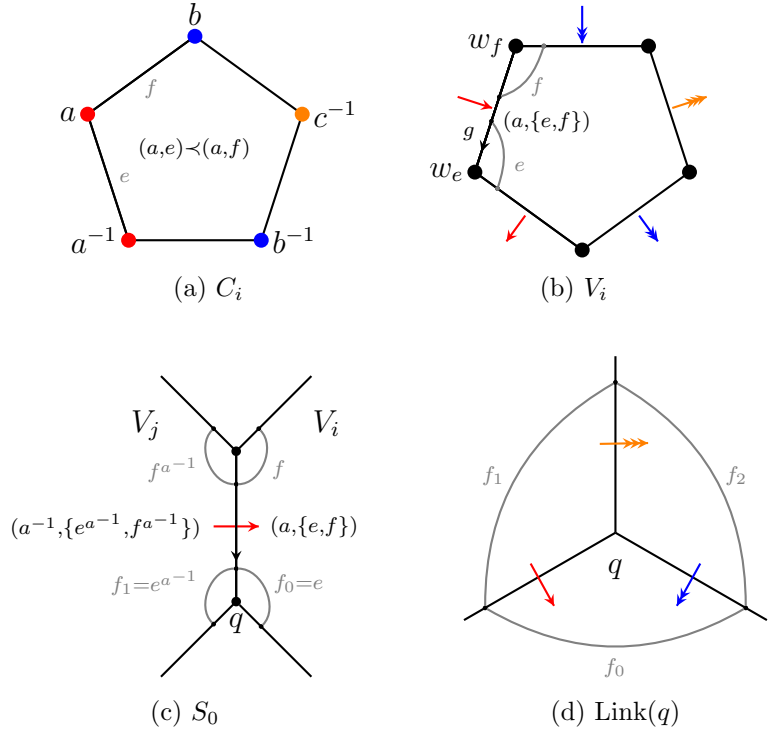
\begin{figure}
  \tikzstyle {a}=[red,postaction=decorate,decoration={%
    markings,%
    mark=at position 1 with {\arrow[red]{stealth};}}]
  \tikzstyle {ab}=[black,postaction=decorate,decoration={%
    markings,%
    mark=at position .85 with {\arrow[black]{stealth};}}]
  \tikzstyle {ab2}=[black,postaction=decorate,decoration={%
    markings,%
    mark=at position .75 with {\arrow[black]{stealth};}}]
  \tikzstyle {b}=[blue,postaction=decorate,decoration={%
    markings,%
    mark=at position .85 with {\arrow[blue]{stealth};},%
    mark=at position 1 with {\arrow[blue]{stealth};}}]
  \tikzstyle {c}=[orange,postaction=decorate,decoration={%
    markings,%
    mark=at position .7 with {\arrow[orange]{stealth};},%
    mark=at position .85 with {\arrow[orange]{stealth};},
    mark=at position 1 with {\arrow[orange]{stealth};}
}]
  \tikzstyle {v}=[draw,shape=circle,fill=black,inner sep=0pt]
  \tikzstyle {gv}=[draw,shape=circle,fill=gray,inner sep=0pt]
  \tikzstyle {bv}=[black,draw,shape=circle,fill=black,inner sep=1pt]
  \tikzstyle{every edge}=[-,draw]
\subfloat[(a) $C_i$]{
	\begin{tikzpicture}[thick]
   		\foreach \i in {0,...,4} 
			\draw (360/5*\i+90:1.5) node [v] (v\i) {} ;
		\draw (0:0) node [] {${}_{(a,e)\prec (a,f)}$};
		\draw (v1) node [left] {$a$} --(v2);
		\draw (v2) node [left] {$a^{-1}$} --(v3) node [right] {$b^{-1}$};
		\draw (v4) node [right] {$c^{-1}$} --(v0);
		\draw (v0) node [above] {$b$} --(v1);
		\draw[] (v0) -- node[pos=.4,below,gray] {${}_f$} (v1);
		\draw[] (v1) -- node[pos=.5,right,gray] {${}_e$} (v2);
		\draw (v4)  --(v3);
		\foreach \i in {1,2}
			\draw (v\i) node [shape=circle,fill=red,inner sep=2pt]  {};
		\foreach \i in {3,0}
			\draw (v\i) node [shape=circle,fill=blue,inner sep=2pt]  {};
		\foreach \i in {4}
			\draw (v\i) node [shape=circle,fill=orange,inner sep=2pt]  {};			
	\end{tikzpicture}
}
\quad
  \subfloat[(b) $V_i$]{	\begin{tikzpicture}[thick]
 		\foreach \i in {0,...,4} {
			\draw (360/5*\i+90:1.9) node [] (w\i) {} ;
			\draw (360/5*\i+90:1.1) node [] (wp\i) {} ;
			\draw (360/5*\i-90:1.5) node [v] (v\i) {} ;
			}
		\draw (v1) --(v2);
		\draw (v2)  --(v3);
		\draw (v4) --(v0);
		\draw (v0) --(v1);
		\draw[ab] (v3)  --(v4);
		\draw[left] (v3) node[] {$w_f$};
		\draw[left] (v4) node[] {$w_e$};
		\draw[] (v3) -- node[right,pos=.6] {${}_{(a,\{e,f\})}$} (v4);
		\draw[] (v3) -- node[left,pos=.7] {${}_{g}$} (v4);
		\draw[] (v2) -- (v3);
		\foreach \i in {1,2}
			\draw (v\i) node [shape=circle,fill=black,inner sep=2pt]  {};
		\foreach \i in {3,0}
			\draw (v\i) node [shape=circle,fill=black,inner sep=2pt]  {};
		\foreach \i in {4}
			\draw (v\i) node [shape=circle,fill=black,inner sep=2pt]  {};	
		\node  [inner sep=0.9pt] at (-90:1.5) {};  
 		\foreach \i in {1} {\draw[a] (w\i) -- (wp\i);};
 		\foreach \i in {2} {\draw[a] (wp\i) -- (w\i);};
 		\foreach \i in {3} {\draw[b] (wp\i) -- (w\i);};
 		\foreach \i in {4} {\draw[c] (wp\i) -- (w\i);};
 		\foreach \i in {0} {\draw[b] (w\i) -- (wp\i);};
		\draw[] (v2) -- node [pos=.8,v] (pf1) {} (v3);
		\draw[] (v4) -- node [pos=.6,v] (pf2) {} (v3);
		\draw[] (v4) -- node [pos=.2,v] (pf3) {} (v0);
		\draw[] (v4) -- node [pos=.4,v] (pf4) {} (v3);
		\draw [gray,out=-90-18,in=-18+36] (pf1) node [v] {}  edge  node[gray,right,pos=.7] {${}_f$} (pf2);			
		\draw [gray,out=72,in=-18-36] (pf3) node [v] {}  edge node[gray,right,pos=.3] {${}_e$} (pf4);			
	\end{tikzpicture}
	}%
\\
\subfloat[(c) $S_0$]{
  \begin{tikzpicture}[thick]
    \draw (-1,2) node[] (v1) {}
	-- node [pos=.7,gv] (fp) {}
	(0,1) node [bv] (v5) {}
	-- node [pos=.3,gv] (f) {}
	(1,2) node [] (v2) {};
	\draw (v5)
	-- (0,-1) node [bv] (v6) {}
	-- node [pos=.3,gv] (ep) {}
	(-1,-2) node [] (v3) {};
	\draw (v6)
	-- node [pos=.3,gv] (e) {}
	(1,-2) node [] (v4) {};
	\draw (v6) node[below] {$q$};
	\node[below] at (-1.2,1.7) {$V_j$};
	\node[below] at (1.2,1.7) {$V_i$};
	\draw[a] (-.3,0) node [] (v7) {}
	-- (.3,0) node [] (v8) {};
	\draw (v7) node [left] {${}_{(a^{-1},\{e^{a^{-1}},f^{a^{-1}}\})}$};
	\draw (v8) node [right] {${}_{(a,\{e,f\})}$};
	\draw[ab2] (v5) {}
	-- node[pos=.2,gv] (u1) {}
	node[pos=.8,gv] (u2) {}
	(v6) {};
	\draw [gray,out=-90-45,in=180] (fp) {}  edge 
	node[gray,left,pos=.6] {${}_{f^{a^{-1}}}$} 
	(u1) {};			
	\draw [gray,out=-45,in=0] (f)  {}  edge 
	node[gray,right,pos=.6] {${}_{f}$} 
	(u1) {};			
	\draw [gray,out=135,in=180] (ep) {}  edge 
	node[gray,left,pos=.6] {${}_{f_1=e^{a^{-1}}}$} 
	(u2) {};			
	\draw [gray,out=45,in=0] (e)  {}  edge 
	node[gray,right,pos=.6] {${}_{f_0=e}$} 
	(u2) {};
\node  [inner sep=0.9pt] at (-90:1.5) {};  
   \end{tikzpicture}
}
\qquad
  \subfloat[(d) $\link(q)$]{\begin{tikzpicture}[thick]
	\foreach \i in {0,...,2} {
		\draw (360/3*\i+90:2.5) node [] (v\i) {} ;
		\draw (360/3*\i+90:2) node [v] (w\i) {} ;
		\draw (0,0) -- (v\i);
		}
		\draw (0,-.2) node [below] {$q$};
		\draw[c] (90 + 10:1.2)--(90-20:1.27);
		\draw[a] (90-10+120 :1.2)--(90 +120 + 20:1.27);
		\draw[b] (90 + 10-120 :1.2)--(90-20-120:1.27);
		\draw [gray,out=-30,in=-150] (w1) edge  node[pos=.5,below] {${}_{f_0}$} (w2);
		\draw [gray,out=-30+120,in=-150+120] (w2) edge  node[pos=.5,right] {${}_{f_2}$} (w0);					
		\draw [gray,out=-30-120,in=-150-120] (w0) edge  node[pos=.5,left] {${}_{f_1}$} (w1);					
		\node  [inner sep=0.9pt] at (-90:1.5) {};  
\end{tikzpicture}}%
\caption{Constructing $V_i$ and $S_0$ from $C_i$ in the proof of Lemma~\ref{lem:equiv}.
In this example, we note from (d) that $f_1=f_0^{a^{-1}}, f_2=f_1^c$ and $f_0=f_2^b$.
}
\label{fig:equiv2}
\end{figure}

By the duality between $C_i$ and $V_i$, each corner of $V_i$ corresponds to an edge in $C_i$.
The link of a vertex $q$ of $S_0$ corresponds to the union of edges in $W(U)$ written as the following sequence
\[f_0\in\delta(x_1^{-1}), f_1 =f_0^{x_1}, f_2 = f_1^{x_2}, \ldots, f_\ell =f_{\ell-1}^{x_{\ell}}\]
so that $f_0 = f_\ell =(\cdots((f_0^{x_1})^{x_2})\cdots)^{x_\ell}$
for some vertices  $x_1,\ldots,x_\ell$ of $W(U)$; see Figure~\ref{fig:equiv2}(c).
By Lemma~\ref{lem:reading}, the word $x_1 \cdots x_\ell$ can be taken as a nontrivial power of a word in $U$.
We will follow the boundary curve $\alpha$ of a small neighborhood of $q$ with some orientation,
and whenever $\alpha$ crosses an edge of $S_0$ with the first component of the label being $a\in\mathcal{A}_n$, we record $a$ if the crossing coincides with the transverse orientation of the edge, and $a^{-1}$ otherwise.
Let $w_q\in F$ be the word obtained by this process.
Then $w_q = x_1 \cdots x_\ell$, up to taking an inverse and cyclic conjugations.

Let $S$ be a surface homeomorphic to $S_0$.
We give $S$ a $2$-dimensional cell complex structure, by letting the homeomorphic image of the dual graph of $S_0^{(1)}$ to be $S^{(1)}$.
In particular, the $2$-cells $P_1,\ldots,P_m$ in $S$ are the connected regions bounded by $S^{(1)}$.
The transverse orientations and the first components of the labels of the sides in $V_1,\ldots,V_t$ induce orientations and labels of the sides of $P_1,\ldots,P_m$.
By duality, the boundary of each $P_i$ in $S$ reads $w_q$ for some vertex $q$ of $S_0$; hence, $\partial P_i$ reads a nontrivial power of a word in $U$.
Finally, if we let $\nu$ be the number of the vertices in $S_0$, then
\[\chi(S) - m = \chi(S_0) - \nu = -\eta + \zeta < 0.\qedhere\]
\end{proof}
 
\begin{REM}\label{rem:polynomial time}
By Lemma~\ref{lem:equiv}, deciding whether a list $U$ of words in a free group is polygonal reduces to a nonnegative integer LP problem;
an alternative description of this fact is given by Touikan~\cite{Touikan2010}.
We note that there is also an algorithm deciding whether $U$ is diskbusting~\cite{Whitehead1936,Stong1997,Stallings1999,RVW2007}.
\end{REM}

Note that Tiling Conjecture for $n$ is equivalent to Conjecture~\ref{con} for $n$. In Sections~\ref{sec:regular} and~\ref{sec:4vertex}, we will prove Conjecture~\ref{con} for $k$-regular lists of words and for rank-two free groups.

\section{Regular Graphs and Proof of Theorem~\ref{thm:main 2}}\label{sec:regular}


We will prove that Conjecture~\ref{con} holds for $k$-regular graphs.
It turns out that we can prove a slightly stronger theorem, which will also be a base case of an induction used in the next section.

\begin{THM}\label{thm:regular}
  Let $G$ be a connected and pairwise well-connected $k$-regular
  Whitehead graph.
  Then there exists a nonempty list of cycles of $G$ 
  with positive integers $m_1$, $m_2$ such that
  \begin{enumerate}[(i)]
  \item   every edge is in exactly $m_1$ cycles in the list,   and
  \item each adjacent pair of edges 
  is contained in exactly $m_2$ cycles in the list.
  \end{enumerate}
  Moreover, at least one of the cycles in this list is not a bigon if
  $G$ has at least $4$ vertices.
\end{THM}

By
Proposition~\ref{prop:whitehead},
Theorem~\ref{thm:main 2} and the following corollary are trivial
consequences of Theorem~\ref{thm:regular}. 

\begin{COR}\label{cor:regularw}
A minimal, diskbusting, $k$-regular list of words in $F_n$ is polygonal when $n>1$.
\end{COR}

A graph $H$ is a \emph{subdivision} of $G$ if 
$H$ is obtained from $G$ by replacing each edge by a path of length at
least one.  
We remark that Conjecture~\ref{con} is also true for all subdivisions of $k$-regular graphs if $k>1$,
because every edge appears the same number of times in Theorem~\ref{thm:regular}.


\medskip

Let us start proving Theorem~\ref{thm:regular}.
A graph $G=(V,E)$ is called a \emph{$k$-graph} if it is $k$-regular
and $\abs{\delta(X)}\ge k$ for every subset $X$ of $V$ with $\abs{X} $ odd.
In particular if $k>0$, then every $k$-graph must have an even number of
vertices, because otherwise $\abs{\delta(V(G))}\ge k$.

It turns out that every $k$-regular
graph 
with the properties required by Conjecture~\ref{con} is a $k$-graph.
\begin{LEM}\label{lem:kgraph}
A $k$-regular pairwise well-connected graph is a $k$-graph.
\end{LEM}
\begin{proof}
Let $G=(V,E)$ be a $k$-regular pairwise well-connected graph.
 Suppose $X\subseteq V$ and $\abs{X}$ is odd. 
  Then there must be $x\in X$ with $x^{-1}\notin X$.
  By the definition of pairwise well-connectedness, 
  $\abs{\delta(X)}\ge k$.
\end{proof}

By the previous lemma, it is sufficient to consider $k$-graphs in order to prove Theorem~\ref{thm:regular}.
By using the characterization of the perfect matching polytope by Edmonds~\cite{Edmonds1965a},
Seymour~\cite{Seymour1979a} showed the following theorem. This is
also explained in Corollary~7.4.7 of the book by Lov\'asz and
Plummer~\cite{LP1986}.
A \emph{matching} is a set of edges in which no two are adjacent. 
A \emph{perfect matching} is a matching meeting every vertex.
\begin{THM}[Seymour~\cite{Seymour1979a}]
  Every $k$-graph is \emph{fractionally $k$-edge-colorable}. 
  In other words, every $k$-graph has a nonempty list of perfect matchings $M_1$, $M_2$,
  $\ldots$, $M_\ell$ such that 
  every edge is in exactly $\ell/k$ of them.
\end{THM}
For sets $A$ and $B$, we write $A\Delta B=(A\setminus B)\cup
(B\setminus A)$.
\begin{LEM}\label{lem:regularlist}
  Let $k>1$.
  Every $k$-graph has a nonempty list of cycles 
 such that
 every edge appears in the same number of cycles
 and 
for each pair of adjacent edges $e$, $f$,
 the number of cycles in the list containing both $e$ and $f$ is
 identical.
\end{LEM}

\begin{proof}
  Let $M_1$, $M_2$, $\ldots$, $M_\ell$ be a nonempty list of perfect
  matchings of a $k$-graph $G=(V,E)$ such that 
  each edge appears in $\ell/k$ of them.
  Then for distinct $i,j$, 
  the set $M_i\Delta M_j$ induces a subgraph of $G$ 
  such that every vertex has degree $2$ or $0$.
  Thus each component of the subgraph $(V,M_i\Delta M_j)$ is a cycle.
  Let $C_1$, $C_2$, $\ldots$, $C_m$  be the list of cycles appearing
  as a component of the subgraph of $G$ induced by $M_i\Delta
  M_j$ for each pair of distinct $i$ and $j$. We allow repeated cycles.
  This list is nonempty because $k>1$ and so there exist $i, j$
  such that  $M_i\neq M_j$.

  Since each edge is contained in exactly $\ell/k$ of $M_1$, $M_2$,
  $\ldots$, $M_\ell$, 
  every edge is in exactly $\frac{\ell}{k}(\ell-\frac{\ell}{k})$
  cycles in the list.
  For two adjacent edges $e$ and $f$, 
  since no perfect matching contains both $e$ and $f$, 
  there are $(\ell/k)^2$ cycles in $C_1$, $C_2$, $\ldots$, $C_m$ using
  both $e$ and $f$. 
\end{proof}

To deduce Theorem~\ref{thm:regular} from Lemmas~\ref{lem:kgraph} and \ref{lem:regularlist}, it only remains to prove that the list $\mathcal{C}$ of cycles we obtain from Lemma~\ref{lem:kgraph} in a given graph $G$ contains at least one cycle which is not a bigon.
Indeed, since $G$ is a connected graph with at least four vertices, it has two adjacent edges $e$ and $f$, not parallel to each other. 
From the conclusion of Lemma~\ref{lem:kgraph}, we see that there must be a cycle in $\mathcal{C}$ containing both $e$ and $f$ and that cycle must have length at least three.

\medskip

We also note that even the minimality assumption can be lifted for rank-two free groups:

\begin{COR}\label{cor:kregular4}
Let $U$ be a $k$-regular list of cyclically reduced words in $F_2$.
Then $U$ is diskbusting if and only if $U$ is polygonal; in this case,
$D(U)$ contains a hyperbolic surface group.
\end{COR}

\begin{proof}
We note that a $k$-regular $4$-vertex graph is always a $k$-graph.

For the sufficiency, we recall that if $U$ is diskbusting in $F_2$,
then $W(U)$ is connected~\cite{Stong1997,Stallings1999}.
Since a connected $4$-vertex graph contains at least one pair of 
incident edges which are not parallel,
Lemma~\ref{lem:regularlist} implies that 
$W(U)$ contains a list of cycles, not all bigons, such that each pair of incident edges appears the same number of times in the list.
Lemma~\ref{lem:equiv} proves the claim.

For the necessity, we note that the proof of the sufficiency part of
Proposition~\ref{prop:tilingequiv} shows if $U$ is not diskbusting in $F_2$, then $D(U)$ does not contain a hyperbolic surface group.
\end{proof}

\section{Proof of Theorem~\ref{thm:main 1}}\label{sec:4vertex}
In general, if we have functions $f\co X\to Y$ and $g\co Y'\to Z$, and an element $x\in X$ such that $f(x)\in Y'$, then we will often use the notation $x^{fg}$ to mean $g(f(x))$.
This is in harmony with our notation of associated connecting maps $e\mapsto e^v$ for Whitehead graphs, where the vertex $v$ is then regarded as a bijection from $\delta(v^{-1})$ to $\delta(v)$. We have already noted that $e^{v v^{-1}} = e$ for $e\in \delta(v^{-1})$.

Let $G$ be a Whitehead graph of a minimal diskbusting list of words in $F_2$.
For a vertex $w$ of $G$, 
a permutation $\pi$ on $\delta(w)$ is called \emph{$w$-good}
if $\{e,e^{\pi w^{-1}}\}$ is a matching of $G$ for every edge $e$
incident with $w$. Note that $\{e,f\}$ is a matching of $G$ if and
only if either $e=f$ or $e$, $f$ share no vertex.
In particular, if $x$ is an edge joining $w$ and $w^{-1}$, then 
$x^{\pi}=x^w$.

For a permutation $\pi$ on $\delta(w)$, let us define a permutation $\tau$ on $\delta(w^{-1})$ by $f^{\tau}= f^{w\pi^{-1}w^{-1}}$.
Then for $e = f^{\tau w}= f^{w \pi^{-1}}$, we have $\{f,f^{\tau w}\} = \{e^{\pi w^{-1}},e\}$.
We see that $\pi$ is $w$-good if and only if $\tau$ is $w^{-1}$-good.

A permutation $\pi$ on  a set $X$ induces a permutation $\pi^{(2)}$ on
$2$-element subsets of $X$ such that
$\{x,y\}^{\pi^{(2)}}=\{x^\pi,y^\pi\}$ for all distinct $x,y\in X$.
A $w$-good permutation $\pi$ on $\delta (w)$ is \emph{uniform}
if
$\pi^{(2)}$ has a list of orbits $X_1$, $X_2$, $\ldots$, $X_t$
satisfying the following.
\begin{enumerate}[(i)]
\item If $\{x,y\}\in X_i$, then $x$ and $y$ do not share a vertex
  other than $w$ or $w^{-1}$ in $G$.
\item 
  There is a constant $c>0$ such that for every edge $e\in \delta(w)$, 
  \[
  \abs{ \{ (X_i,F): 1\le i\le t,~F\in X_i \text{ and } e\in F\} } = c.\]
\end{enumerate}

The following lemma shows that 
in order to prove Conjecture~\ref{con} for $4$-vertex graphs, 
it is enough to find a $w$-good uniform
permutation on the edges incident with a vertex $w$ of minimum degree.

\begin{LEM}\label{lem:inductive}
Let $G$ be a connected and pairwise well-connected Whitehead graph on
four vertices.
Let $w$ be a vertex of $G$ with the minimum degree.
If there is a $w$-good uniform permutation $\pi$ on
  $\delta(w)$, 
  then 
  $G$ admits a nonempty list of cycles satisfying the following properties.
  \begin{enumerate}[(a)]
  \item   For distinct edges $e_1,e_2\in \delta(w)$, 
    the number of cycles in the list containing both $e_1$ and $e_2$
    is equal to the number of cycles in the list containing both
    $e_1^{w^{-1}}$ and $e_2^{w^{-1}}$.
  \item There is a constant $c_1>0$ such that each edge appears in
    exactly $c_1$ cycles in the list.
 \item  There is a constant $c_2>0$ such that 
    for a vertex $v\in V(G)\setminus\{w,w^{-1}\}$ and each pair of distinct edges $e_1,e_2\in \delta(v)$,
    exactly $c_2$ cycles in the list contain both $e_1$ and $e_2$.
  \item The list contains a cycle of length at least three.
\end{enumerate}
\end{LEM}
For $e\in E(G)$, we denote by $G\setminus e$ the graph obtained from
$G$ by deleting 
the edge $e$. Ends of $e$ are not removed when deleting $e$.
\begin{proof}
We say that a list of cycles is \emph{desirable} if it satisfies (a), (b), (c), and (d).
Let $u$ be a vertex of $G$ other than $w$ and $w^{-1}$.
We proceed by induction on $\deg(u)-\deg(w)$.
If $\deg(u)=\deg(w)$, then the conclusion follows by
  Theorem~\ref{thm:regular}.
We now assume that $\deg(u)>\deg (w)$.
There should exist an edge $e$ joining $u$ and $u^{-1}$.

Since the statement of the lemma is independent of the choice of the associated connecting map from $\delta(u)$ to $\delta(u^{-1})$,
we may assume $e^u=e$ so that $G\setminus e$ is a Whitehead
graph.

Observe that $G\setminus e$ is connected and pairwise well-connected Whitehead graph
and $\pi$ is also a $w$-good uniform permutation in $G\setminus e$.
By the induction hypothesis, $G\setminus e$ has  a desirable list of cycles 
$\mathcal{C}'=(C_1',C_2', \ldots,C_s')$.
Let $c_1'$, $c_2'$ be the constants given by (b) and (c), respectively, for the list $\mathcal{C}'$ in $G\setminus e$.
Since $\pi$ is $w$-good uniform, $\pi^{(2)}$ has a list of
  orbits $X_1$, $X_2$, $\ldots$, $X_t$ satisfying (i) and (ii),
  where each edge in $\delta(w)$ appears $c$ times in this list.
 
  Suppose that $\{x,y\}\in X_i$. Then $\{x^\pi,y^\pi\}\in
  X_i$.
  If $x,y\in \delta(w^{-1})$, then
  we let $C_{xy}$ be a cycle formed by two edges $x=x^{\pi w^{-1}}$
  and $y=y^{\pi w^{-1}}$.
  If $x,y\notin \delta(w^{-1})$,
  let $C_{xy}$ be a list of two cycles, one  formed by three edges
  $e$, $x$, $y$, and 
  the other  formed by three edges $e$, $x^{\pi w^{-1}}$, $ y^{\pi w^{-1}}$.
  If exactly one of $x$ and $y$, say $y$,  is incident with $w^{-1}$, then
  let $C_{xy}$ be the cycle formed by four edges $e$, $x$,
  $y= y^{\pi w^{-1}}$, $x^{\pi w^{-1}}$.
  Since $x$ and $y$ never share $u$ or $u^{-1}$ by (i), 
  the list $C_{xy}$ always consists of one or  two cycles of $G$.
    
  Let $\mathcal{C}=(C_1,C_2,\ldots,C_p)$ be the list of all cycles in $C_{xy}$ for each
  member $\{x,y\}$ of $X_i$ for all $i=1,2,\ldots,t$. Notice that we allow
  repetitions of cycles.
  We claim that $\mathcal{C}$ satisfies (a).
 For each occurrence of $x,y\in \delta(w)$ in a cycle in the
  list, 
  there is a corresponding $i$ such that $\{x,y\}\in X_i$. Since $X_i$ is an
  orbit, there is $\{x',y'\}\in X_i$ where $(x')^\pi=x$ and $(y')^\pi=y$.
  Then the list contains cycles in $C_{x'y'}$ for $X_i$.
  This proves the claim because
$x^{w^{-1}}=(x')^{\pi w^{-1}}$ and $y^{w^{-1}}=(y')^{\pi w^{-1}}$.
 
  By (ii) of the definition of a uniform permutation, 
  for each edge $f$ incident with $w$, 
  there are $c$ cycles in the list $\mathcal{C}$ 
  containing the edge $f$ of $G$. 
 Notice that whenever an edge $f$ in
  $C_{xy}$  is in   $\delta(\{w,w^{-1}\})$, $C_{xy}$
  contains $e$ and $f^{\pi w^{-1}}$ by the construction.
  Therefore  every edge incident with $w$ or $w^{-1}$
  appears $c$ times 
 in $\mathcal{C}$.

  We now construct a desirable list of cycles for $G$ as follows:
  We take $c_2'$ copies of $\mathcal{C}$,
  $c$ copies of $\mathcal{C}'$, 
  and $cc_2'$ copies of cycles formed by  $e$ and each edge
  $f\neq e$  joining $u$ and $u^{-1}$.
  We claim that this is a desirable list of cycles of $G$.
  It is trivial to check (a) and (d). 
 For distinct edges $e_1,e_2$ incident with $u$, 
  the list contains $cc_2'$ cycles containing both of them, verifying (c).
  Let $a$ be the number of edges in $\delta(u)$ incident with $w$ or $w^{-1}$
  and let $b$ be the number of edges joining $u$ and $u^{-1}$.
  By (c) on $G\setminus e$, we have $c_1'=c_2'(a+b-2)$.
  Finally to prove (b),  every edge incident with $w$
  or $w^{-1}$ appears $cc_2'+cc_1'=cc_2'(a+b-1)$ times in the list
  and 
  the edge $e$ appears $acc_2'+(b-1)cc_2'=cc_2'(a+b-1)$ times in the list.
  An edge $f\neq e$ joining $u$ and $u^{-1}$ appears
  $cc_1'+cc_2'=cc_2'(a+b-1)$ times.
\end{proof}

\begin{REM}
Let us describe topological ideas of the lemma for $U\subseteq F_2$ and $G=W(U)$.
Recall that a corner and a link in a $U$-polygonal surface is realized as an edge and a cycle, respectively in $G$.
For each \[\{x,y\}\subseteq {\delta(w)\choose 2}\setminus\left({\delta(u)\choose 2}\cup{\delta(u^{-1})\choose 2}\right)\]
we are constructing a cycle (or two cycles) $C_{xy}$ in $G$
and ultimately, a list of links that are matched up to build a $U$-polygonal surface $S$
as in the proof of Lemma~\ref{lem:equiv}.
In that $U$-polygonal surface, $C_{xy}$ will correspond to the link(s) of one or two vertices,
and at those links $\{x,y\}$ and $\{x^{\pi w^{-1}},y^{\pi w^{-1}}\}$ are the pairs of corners at the incoming and outgoing portions of $w$-edge, respectively.
Having a $\pi^{(2)}$-orbit guarantees that $\{x^{\pi},y^{\pi}\}$
appears as the pair of corners at the incoming portion at another vertex in $S$ as well,
so that we can match the corners $\{x^{\pi w^{-1}},y^{\pi w^{-1}}\}$ to the corners  $\{x^{\pi},y^{\pi}\}$
along $w$-edges; see Figure~\ref{fig:equiv2}(c) for a dual picture.
So far we have matched the pairs of corners at incoming or outgoing $w$-edges.
The above lemma shows how to match the rest of the pairs of the corners in a ``uniform'' way.
That is, each pair of edges at $\delta(u)$ or at $\delta(u^{-1})$ appear 
exactly $c_2$ times among the links in $S$.
\end{REM}

\subsection{Lemma on odd paths and even cycles}
To find a $w$-good uniform permutation of $\delta(w)$,
we need a combinatorial lemma on a disjoint union of
odd directed paths and even directed cycles.
The \emph{length} of a path or a cycle is the number of its edges.
A path or  a cycle is \emph{odd} if its length is odd
and \emph{even} otherwise.
We say that a directed path or a directed cycle is \emph{long} if its length is at least three, and \emph{short} otherwise.

\begin{DEFN}\label{defn:calm}
Let $D$ be a directed graph 
that is a disjoint union of odd directed paths and even directed cycles
such that all ends of the directed paths are colored with red or blue
and no other vertices are colored.
\begin{enumerate}
\item A vertex of $D$ is called a \emph{source} if its in-degree is $0$. It is called a \emph{sink} if its out-degree is $0$.
\item We say $D$ is \emph{calm} if at most half of all  the vertices are blue and at most half of all the vertices are red.
\item 
  If a directed path $P$ starts with a red vertex and ends
  with a red vertex, then we say $P$ is an \emph{R-R path}. Similarly we define \emph{R-B paths}, \emph{B-B paths} and \emph{B-R paths}.
\item 
  A set of directed paths is \emph{monochromatic} if it has no blue vertex or no red vertex.
\end{enumerate}
\end{DEFN}

\begin{LEM}\label{lem:oddpath}
  Let $D$ be a directed graph that is a disjoint union of
  odd directed paths and even directed cycles such that all sources
  or sinks are colored with red or blue, 
  no other vertices are
  colored,
  and the number of red sources is equal to the number of red sinks.
  If $D$ has at least four vertices and is calm,
 then
  $D$ can be partitioned into calm subgraphs,
  each of which is one of eight types listed below. (See Figure~\ref{fig:types}.)
 \begin{enumerate}
 \item A short R-R  path, 
   a short B-B path, and possibly a short cycle.
\item A monochromatic path and one or two  short cycles.
 \item A short cycle, a B-R path, and an R-B path.
 \item At least two short cycles.
 \item  A long monochromatic path 
   and a monochromatic set of short paths, possibly none.
 \item 
   A B-R path, an R-B
   path, and a monochromatic set of short paths, possibly none.
 \item  A long cycle 
   and a monochromatic set of short paths, possibly none.
 \item A long cycle and a short cycle.
 \end{enumerate}
\end{LEM}
We remark that in a subgraph of type (5), we require that 
the long path is monochromatic and the set of short
paths monochromatic, but we allow the long path to have a 
color unused in short paths.
We also emphasize that each subgraph of type (1)-(8) should be
calm. For instance,  in a subgraph of type (6), the sum of lengths of
an R-B path and a B-R
path should be big enough to make this graph calm.
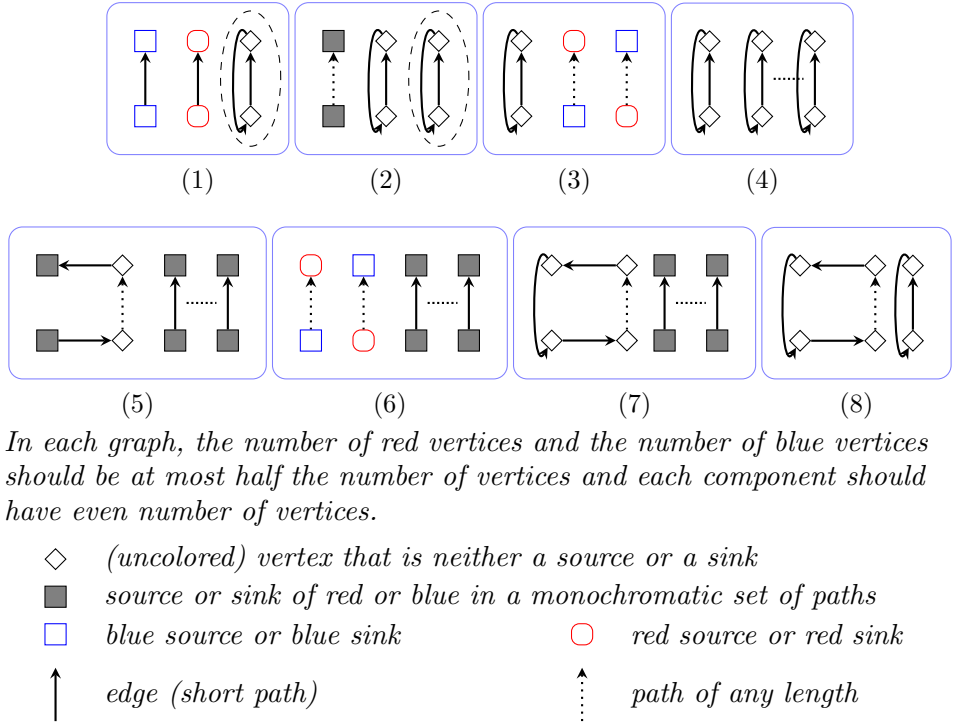
\begin{figure}
  \tikzstyle{v}=[draw, shape=diamond,solid, color=black, inner sep=2pt, minimum width=2pt]
  \tikzstyle{r}=[rounded corners=3,draw, solid, color=red, inner sep=4pt, minimum width=4pt]
  \tikzstyle{re}=[ rounded corners=3,draw, solid, color=red, fill=red!50, inner sep=4pt, minimum width=4pt]
  \tikzstyle{b}=[draw, solid, color=blue, inner sep=4pt, minimum width=4pt]
  \tikzstyle{be}=[draw,solid, color=blue, fill=blue!50, inner sep=4pt, minimum width=4pt]
  \tikzstyle{g}=[draw, solid, color=black, fill=black!50, inner sep=4pt, minimum width=4pt]
  \tikzstyle{ge}=[draw,solid, color=black, fill=black!50, inner sep=4pt, minimum width=4pt]
  \tikzstyle{every edge}=[->,>=stealth,draw,thick,overlay]
  \tikzstyle{w}=[node distance=7mm]
  \tikzstyle{ddd}=[thick,densely dotted,-]
  \tikzstyle{e}=[->,thick,color=black,out=120,in=-120,draw,>=stealth]
  \tikzstyle{background rectangle}=[draw=blue!50, rounded corners=1ex]
  \tikzstyle{showbg}=[show background rectangle,inner frame sep=2ex]
  \subfloat[(1)]{\tikz[showbg]{
      \node [b] (bl) {};
      \node [b] [above of=bl]{}
      edge[<-] (bl);
      \node [r] (br) [w,right of=bl] {};
      \node [r] [above of=br] {}
      edge[<-] (br);
      \node [v] (r) [w,right of=br] {};
      \node [v] [above of=r] {}
      edge[<-] node[midway] (ct){} (r)
      edge[e] (r);   
      \draw [overlay,dashed] (ct) ellipse (.4cm and .9cm);
  }}\,%
  \subfloat[(2)]{\tikz[showbg]{
    \node[g] (bl){};
    \node[ge]  [above of=bl] {}
      edge[<-,dotted] (bl);
      \node [v] (r) [w,right of=bl] {};
      \node [v] [above of=r] {}
      edge[<-] (r)
      edge[e] (r);   
     \node [v] (r2) [w,right of=br] {};
      \node [v] [above of=r2] {}
      edge[<-] node[midway] (ctt){} (r2)
      edge[e] (r2);   
      \draw [dashed,overlay] (ctt) ellipse (.4cm and .9cm);
 }}\,%
\subfloat[(3)]{\tikz[showbg]{
      \node[r] (bl)  {};
      \node[b] (b1) [above of=bl] {}
      edge[<-,dotted] (bl);
      \node[b] (bll) [w,left of=bl]{};
      \node[r] (b1) [above of=bll] {}
      edge[<-,dotted] (bll);
      \node [v] (r) [w,left of=bll] {};
      \node [v] [above of=r] {}
      edge[<-] (r)
      edge[e] (r);   
   }}\,%
 \subfloat[(4)]{\tikz[showbg]{
      \node [v] (r) [w] {};
      \node [v] [above of=r] {}
      edge[<-] (r)
      edge[e] (r);   
      \node [v] (r1) [w,right of=r] {};
      \node [v] [above of=r1] {}
      edge[<-] node[midway] (start) {} (r1)
      edge[e] (r1);   
      \node [v] (r2) [w,right of=r1] {};
      \node [v] [above of=r2] {}
      edge[<-] node[midway] (end) {} (r2)
      edge[e] (r2);   
     \draw [ddd] (start)--(end);
    }}\\%
 \subfloat[(5)]{\tikz[showbg]{
    \node[g] (bl){};
      \node[v] (b1) [right of=bl] {}
      edge[<-] (bl);
      \node[v] (b2) [above of=b1] {}
      edge[<-,dotted] (b1);
      \node[ge]  [left of=b2] {}
      edge[<-] (b2);
      \node [g] (br) [w,right of=b1] {};
      \node [ge] [above of=br] {}
      edge[<-] node[midway](start) {} (br);
      \node [g] (br2) [w,right of=br] {};
     \node [ge] [above of=br2] {}
      edge[<-] node[midway](end) {} (br2);
      \draw [ddd] (start)--(end);
  }}\,%
 \subfloat[(6)]{\tikz[showbg]{
      \node[r] (bl)  {};
      \node[b] (b1) [above of=bl] {}
      edge[<-,dotted] (bl);
      \node[b] (bll) [w,left of=bl]{};
      \node[r] (b1) [above of=bll] {}
      edge[<-,dotted] (bll);
      \node [g] (br) [w,right of=bl] {};
     \node [ge] [above of=br] {}
      edge[<-] node[midway](start) {} (br);
      \node [g] (br2) [w,right of=br] {};
     \node [ge] [above of=br2] {}
      edge[<-] node[midway](end) {} (br2);
      \draw [ddd] (start)--(end);
    }}\,%
  \subfloat[(7)]{\tikz[showbg]{
    \node[v] (bl){};
      \node[v] (b1) [right of=bl] {}
      edge[<-] (bl);
      \node[v] (b2) [above of=b1] {}
      edge[<-,dotted] (b1);
      \node[v]  [left of=b2] {}
      edge[<-] (b2)
      edge[e] (bl);
     \draw (1.5,0) node [g] (br) [w] {};
     \node [ge] [above of=br] {}
      edge[<-] node[midway](start) {} (br);
      \node [g] (br2) [w,right of=br] {};
     \node [ge] [above of=br2] {}
      edge[<-] node[midway](end) {} (br2);
      \draw [ddd] (start)--(end);
      }}\,%
  \subfloat[(8)]{\tikz[showbg]{
    \node[v] (bl){};
      \node[v] (b1) [right of=bl] {}
      edge[<-] (bl);
      \node[v] (b2) [above of=b1] {}
      edge[<-,dotted] (b1);
      \node[v]  [left of=b2] {}
      edge[<-] (b2)
      edge[e] (bl);
     \draw (1.5,0) node [v] (br) [w] {};
     \node [v] [above of=br] {}
     edge[<-] (br)
     edge[e] (br);
     }}
   
   \begin{flushleft}
     \noindent\small\emph{%
       In each graph, 
       the number of red vertices and the number of blue vertices
       should be  at
       most 
       half the number of vertices
     and each component should have even number of vertices.}
   \end{flushleft}

  \begin{tikzpicture}
     \draw (0,1.5) node [g] {};
     \draw(.5,1.5) node [right]{\small\emph{source or sink of red or blue
       in a monochromatic set of paths}};
     \draw (0,2) node[v] {};
     \draw (.5,2) node  [right] {\small\emph{(uncolored) vertex that is neither a
       source or a sink}};
     \draw (7,1) node[r]   {} ; 
     \draw (7.5,1) node [right] {\small\emph{red source or red sink}};
     \draw (0,1) node[b] {};
     \draw (.5,1) node [right] {\small\emph{blue source or blue sink}};
     \draw (0,-.3) node  (br3)  {};
     \node  [above of=br3] {}
      edge[<-]  (br3);
     \draw (0.5,0.2) node[right]  {\small\emph{edge (short path)}};
     \draw (7,-.3) node  (br2)  {};
     \node  [above of=br2] {}
      edge[<-,dotted]  (br2);
     \draw (7.5,0.2) node[right]  {\small\emph{path of any length}};
     \end{tikzpicture}

  \caption{Eight types of calm subgraphs in Lemma~\ref{lem:oddpath}}
  \label{fig:types}
\end{figure}
\begin{proof}
  We proceed by induction on $\abs{V(D)}$.
  If $D$ has a subgraph $H$ that is a disjoint union of a short R-R
  path and a short B-B path,
  then $D\setminus V(H)$, the subgraph obtained by removing vertices
  of  $H$ from  $D$, 
  is still calm. 
  If $D=H$, then we have nothing to prove.
  If $\abs{V(D)\setminus V(H)}= 2$, then 
  $D$ is the disjoint union of a short R-R path, a short B-B path, and
  a short cycle, and therefore $D$ is a directed graph of type (1). 
  If $\abs{V(D)\setminus V(H)}\ge 4$, then 
  $H$ is a calm subgraph of type (1).
  Then we apply the induction
  hypothesis to get a partition for $D\setminus V(H)$. 

  Therefore we may assume that $D$ has 
  no pair of a short B-B path and a short R-R path.
  By symmetry, we may assume that $D$ has no short R-R path.
  Then in each component, the number of red vertices is at most half
  of the number of vertices. Thus, in order to check whether some
  disjoint union of components is calm, it is enough to count blue vertices.

  Suppose that $D$ has a short cycle and a short B-B path.
  We are done if $D$ is a graph of type (2).
  Thus we  may assume that $D$ has at least  
eight vertices. 
  Let $X$ be the set of vertices in the pair of a short cycle and a short B-B
  path. Then the subgraph of $D$ induced on $X$ is a subgraph of type (2).
  Because $X$ has two blue vertices and two uncolored vertices,
  $D\setminus X$ is calm and has at least 
four  vertices.
  By the induction hypothesis, we obtain a calm partition of
  $D\setminus X$. This together with the subgraph induced by $X$ is 
  a calm partition of $D$.

  We may now assume that either $D$ has no short cycles, or $D$ has no
  short B-B path. 

  \noindent (Case 1) Suppose that $D$ has no short cycles. 
  The subgraph of $D$ consisting of all components other than short B-B paths can be partitioned
  into calm subgraphs   $P_1$, $P_2$, $\ldots$, $P_k$ of type (5),
  (6), or (7), 
because the number of R-B paths is equal to the number of B-R paths.
We claim that short B-B paths can be
  assigned to   those subgraphs while maintaining each $P_i$ to 
  be calm.
 Suppose that $P_i$ has $2b_i$ blue vertices and $2n_i=\abs{V(P_i)}$.
 Notice that $b_i$ and $n_i$ are integers.
 Let $x$ be the number of short B-B paths in $D$.
 Since $D$ is calm,
 $2(2x+\sum_{i=1}^k 2b_i)  \le \sum_{i=1}^k 2n_i+2x$
and therefore $x\le \sum_{i=1}^k (n_i-2b_i)$.
 Each $P_i$ can afford to have  $n_i-2b_i$ short B-B paths
 to be calm.
 Overall all $P_1,\ldots,P_k$ can afford $\sum_{i=1}^k (n_i-2b_i)$ short B-B
 paths; thus consuming all short B-B paths.
 This proves the claim.

 \noindent (Case 2) 
Suppose $D$ has short cycles but has no short B-B paths.
If $D$ has at least two short cycles, then we can take all short
 cycles as a subgraph of type (4)
 and the subgraph of $D$ consisting of all components other than short cycles can be decomposed into
 subgraphs, each of which is type (5), (6), or (7).

Thus we may assume $D$ has exactly one short cycle.
Since $D$ has at least 
four vertices, $D$ must have 
a subgraph $P$ consisting of components of $D$ that is  one of the following type:
a monochromatic path,
a long cycle,
or a pair of a B-R path and an R-B path.
Then $P$ with the short cycle forms a calm subgraph of type (2), (8),
or (3), respectively.
 The subgraph of $D$ induced by all the remaining components 
 can be decomposed into subgraphs of type (5), (6), and (7).
\end{proof}

\subsection{Finding a $w$-good uniform permutation} 
Let $G$ be a connected and pairwise well-connected Whitehead graph with four vertices.
Let $w$ be a vertex of $G$ with the minimum degree
and let $u$ be a vertex of $G$ other than $w$ and $w^{-1}$. 

Let $e_1,e_2,\ldots,e_m$ be the edges incident with $w$
and let $f_1,f_2,\ldots,f_m$ be the edges incident with $w^{-1}$
so that $f_i^w=e_i$.
We construct an auxiliary directed graph $D$ on 
the disjoint union of $\{e_1,e_2,\ldots,e_m\}$
and $\{f_1,f_2,\ldots,f_m\}$ as follows:
\begin{enumerate}[(i)]
\item For all $i\in \{1,2,\ldots,m\}$, 
$D$ has an edge from $f_i$ to $e_i$.
\item If $e_i$ and $f_j$ denote the same edge in $G$, 
  then $D$ has an edge from $e_i$ to $f_j$.
\end{enumerate}
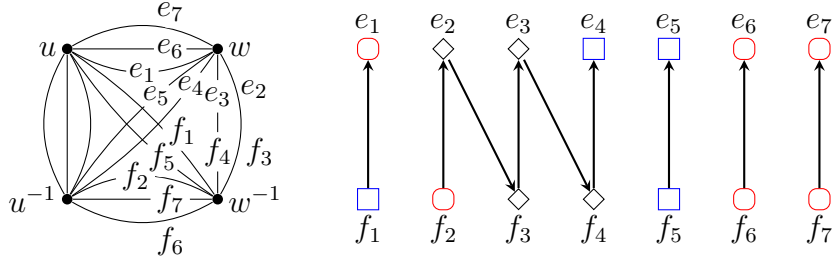
\begin{figure}
  \centering
  \tikzstyle{v}=[circle,solid, solid, fill=black, inner sep=0pt, minimum width=4pt]
 \tikzstyle{every edge}=[-,draw]
  \tikzstyle{l}=[circle,inner sep=0pt,fill=white]
  \begin{tikzpicture}
    \node [v] at (-3,0) (v0) {};
    \node [v] at (-3,2) (v1) {} 
    edge (v0)
    edge [bend right] (v0)
    edge [bend left] (v0);
 \node [v] at (-1,0) (v2) {}
    edge node [l,pos=0.3] {$f_7$} (v0)
    edge [bend right] node [l,pos=0.55] {$f_2$} (v0)
    edge [bend left] node [l,pos=0.3,below=0.5pt] {$f_6$} (v0)
    edge [in=-45-10,out=180-45+10] node [l,pos=0.3] {$f_5$} (v1)
   edge [in=-45+10,out=180-45-10] node [l,pos=0.25,above] {$f_1$} (v1);
   \node [v] at (-1,2) (v3) {}
   edge node [l,pos=0.30] {$e_3$} node [l,pos=0.7] {$f_4$} (v2)
    edge [bend left] node [l,right=.5pt,pos=0.25] {$e_2$} node [right=0.5pt,l,pos=0.7] {$f_3$} (v2)
    edge [out=180+45+10,in=45-10] node [l,pos=.18] {$e_4$} (v0)
    edge [out=180+45-10,in=45+10] node [l,pos=.34] {$e_5$} (v0)
    edge [bend left]  node [l,pos=0.5] {$e_1$} (v1)
    edge  node [l,pos=0.3] {$e_6$} (v1)
    edge [bend right] node [l,above,pos=0.3] {$e_7$} (v1);
    \draw (v0) node [left] {$u^{-1}$};
    \draw (v1) node [left] {$u$};
    \draw (v2) node [right] {$w^{-1}$};
    \draw (v3) node [right] {$w$};
  \tikzstyle{box}=[draw, shape=diamond,solid, color=black, inner sep=2pt, minimum width=2pt]
    \tikzstyle{r}=[ rounded corners=3,draw, solid, color=red, inner sep=4pt, minimum width=4pt]
    \tikzstyle{re}=[ rounded corners=3,draw, solid, color=red, fill=red!50, inner sep=4pt, minimum width=4pt]
    \tikzstyle{b}=[draw, solid, color=blue, inner sep=4pt, minimum width=4pt]
    \tikzstyle{be}=[draw, solid, color=blue, fill=blue!50, inner sep=4pt, minimum width=4pt]
    \tikzstyle{every edge}=[->,>=stealth,draw,thick]
    \tikzstyle{eb}=[->,thick,color=blue,out=120,in=-120,draw,>=stealth]
    \tikzstyle{er}=[->,thick,color=red,out=120,in=-120,draw,>=stealth]
    \foreach \x in {1,6,7}  {    \node [r] at (\x,2) (e\x) {} ;}
   \foreach \x in {4,5}  {    \node [b] at (\x,2) (e\x) {}; };
    \foreach \x in {2,3}  {    \node [box] at (\x,2) (e\x) {}; };
    \foreach \x in {2,6,7}  {    \node [r] at (\x,0) (f\x) {} edge (e\x);   };
    \foreach \x in {1,5}  {    \node [b] at (\x,0) (f\x) {} edge (e\x); };
    \foreach \x in {3,4}  {    \node [box] at (\x,0) (f\x) {} edge
      (e\x); };
    \foreach \x in {1,2,3,4,5,6,7} {
      \node [above=2pt] at (e\x)  {$e_{\x}$};
      \node [below=2pt] at (f\x)   {$f_{\x}$};
    }
    \node at (e2) {} edge (f3);
    \node at (e3) {} edge (f4);
\end{tikzpicture} 
  \caption{A graph and its auxiliary directed graph at $w$}
  \label{fig:ex}
\end{figure}
We have an example in Figure~\ref{fig:ex}.
It is easy to observe the following.
\begin{itemize}
\item Every vertex in $\{e_1,e_2,\ldots,e_m\}$ of $D$ 
  has in-degree $1$.
\item Every vertex in $\{f_1,f_2,\ldots,f_m\}$ of $D$ 
  has out-degree $1$.
\item A vertex $e_i$ of $D$ has out-degree $1$ if the edge $e_i$ of
  $G$ is incident with $u^{-1}$,
  and out-degree $0$ if otherwise.
\item A vertex $f_i$ of $D$ has in-degree $1$ if the edge $f_i$ of
  $G$ is incident with $u$,
  and in-degree $0$ if otherwise.
\end{itemize}
By the degree condition, $D$ is the disjoint union of odd
directed paths
and even directed cycles.

Let $r$ be the number of edges of $G$ joining $u$ and $w$
and let $b$ be the number of edges of $G$ joining $u^{-1}$ and $w$.
For each $i$, we color $e_i$ 
red 
if it is incident with $u$
and 
blue
if it is incident with $u^{-1}$.
Similarly for each $i$, we color $f_i$ 
blue
if it is incident with $u$
and 
red 
if it is incident with $u^{-1}$.
Clearly there are $r$ red vertices and $b$ blue vertices in $\{e_1,e_2,\ldots,e_m\}$.
By general observation on Whitehead graphs, 
there are $r$ edges of $G$ joining $u^{-1}$ and $w^{-1}$,
and $b$ edges of $G$ joining $u$ and $w^{-1}$. 
We also assume that $G$ has $\deg(u)$ edge-disjoint paths
from $u$ to $u^{-1}$.
Therefore $\abs{\delta (\{u,w\})}\ge \abs{\delta(u)}$
and $\abs{\delta(\{u,w^{-1}\})}\ge \abs{\delta(u)}$.
This implies that 
$b+b+(m-r-b)\ge b+r$ and 
$r+r+(m-r-b)\ge b+r$.
Thus 
\[
2r\le m \text{ and } 2b\le m.
\]
Note that these inequalities are consequences of minimal diskbusting property of $U\subseteq F_2$ if $G=W(U)$.

From now on, our goal is to find a $w$-good permutation $\pi$ on $\delta(w)$ 
given a directed graph $D$ with additional edges.

\begin{LEM}\label{lem:good}
  Suppose that there exists a directed graph $D'$ 
  obtained from $D$ by adding one edge 
  from each sink 
  to a source 
   with the same color
  so that every vertex has in-degree $1$ and out-degree $1$ in $D'$. 
  Let $\pi$ be a permutation on $\delta (w)=\{e_1,e_2,\ldots,e_m\}$
  so that 
  $e_i^\pi=e_j$ if and only if 
  $D'$ has a directed walk from $e_i$ to $e_j$ of length two.
  Then $\pi$ is $w$-good.
\end{LEM}
Let us call such a directed graph $D'$ a \emph{completion} of $D$.
A completion of $D'$ always exists, 
because the number of red sources is equal to the
number of red sinks.
Clearly there are $r! \, b! $ completions  of $D$.
\begin{proof}
  It is enough to show that 
  if $D'$ has an edge $e$ from $e_i$ to $f_j$, 
  then $\{e_i,f_j\}$  is a matching of $G$.
  If 
  $e\in E(D)$, then
  $e_i=f_j$ and therefore $\{e_i,f_j\}=\{e_i\}$ is a matching of $G$.
  If 
  $e\notin E(D)$, then 
  $e_i$ and $f_j$ should have the same color 
  and therefore $e_i$ and $f_j$ do not share any vertex.
\end{proof}

Out of $r!\,b!$ completions of $D'$, we wish to find a completion $D'$ of $D$  so that the
$w$-good permutation induced by $D'$ is uniform.

\begin{LEM}\label{lem:uniformtype}
  If $D$ is a calm directed graph of type (1), (2), $\ldots$, (8) described
  in Lemma~\ref{lem:oddpath}, then $D$ has a completion $D'$ so that 
  the induced $w$-good permutation is uniform.
\end{LEM}
\begin{proof}
  We claim that for each type of a directed graph $D$, there is 
  a completion $D'$ of $D$ such that 
  its induced $w$-good permutation $\pi$ on $\delta(w)$ is uniform. Recall that a
  $w$-good permutation $\pi$ is uniform if 
  $\pi^{(2)}$ has a list of orbits $X_1$, $X_2$, $\ldots$, $X_t$
  satisfying the following conditions:
  \begin{enumerate}[(i)]
  \item If $\{x,y\}\in X_i$, then $x$ and $y$ do not share a vertex
    other than $w$ or $w^{-1}$ in $G$.
  \item 
    There is a constant $c>0$ such that for every edge $e\in \delta(w)$, 
    \[
    \abs{ \{ (X_i,F): 1\le i\le t,~F\in X_i \text{ and } e\in F\} } = c.\]
  \end{enumerate}

  \medskip\noindent Case 1: Suppose that $D$ is of type (1) or (4) with $k$ components.
  Then There is a unique completion $D'$ of $D$.
  It is easy to verify that the list of all orbits of $\pi^{(2)}$
  satisfies the conditions (i) and (ii) where $c=k-1$.

  \medskip\noindent Case 2: 
  Suppose that $D$ is of type (2).
  Then $D$ consists of a monochromatic path $P$ and one or two
  short cycles. 
  A completion $D'$ of $D$ is unique, as it is obtained by adding an
  edge from the terminal vertex of $P$ to the initial vertex of $P$.
Let $\pi$ be the permutation of $\delta(w)$ induced by $D'$.
  Let $x_1,x_2,\ldots,x_m$ be the edges in $\delta(w)$ that are in
  $P$ such that $x_i^\pi=x_{i+1}$ for all $i=1,2,\ldots,m$ where $x_{m+1}=x_1$.
  Let $y_1\in \delta(w)$ be the vertex in
  the first short cycle such that $y_1^\pi=y_1$. If $D$ has two
  cycles, then let $y_2\in\delta(w)$
  be the vertex in the second short cycle such that $y_2^\pi=y_2$.

  Then 
  $O_j=\{\{x_i,y_j\}: 1\le i\le m\}$ is an orbit of $\pi^{(2)}$
  satisfying (i).
  If $m>1$, then $O_P=\{\{x_i,x_{i+1}\}: 1\le i\le m\}$ is an orbit of $\pi^{(2)}$
  satisfying (i) in which each $x_i$ appears twice if $m>2$ and each
  $x_i$ appears once if
  $m=2$.

  If $D$ has only one cycle, then each $x_i$ appears once and $y_1$
  appears $m$ times in $O_1$. 
  So if $m=1$, then $O_1$ satisfies (i) and (ii).
  If $m=2$, then $O_1$ and $O_P$ form a list of orbits of $\pi^{(2)}$
  satisfying (i) and (ii). If $m>2$, then a list of two copies of $O_1$ and
  $(m-1)$ copies of $O_P$ satisfies (i) and (ii).

If $D$ has two short cycles, then in $O_1$ and $O_2$, 
  each $x_i$ appears twice and each $y_j$ appears $m$ times.
Notice that $\{\{y_1,y_2\}\}$ is an orbit of $\pi^{(2)}$.
  If $m=1$, then a list of $O_1$, $O_2$, and $\{\{y_1,y_2\}\}$ satisfies
  (i) and (ii).
 If $m=2$, then a list of $O_1$ and $O_2$ satisfies (i) and (ii). 
  If $m>3$, then 
  a list of 
 two copies of $O_1$,
 two copies of $O_2$, 
 and $(m-2)$ copies of $O_P$ satisfies (i) and (ii).
  
 \medskip\noindent Case 3: If $D$ is of type (3), then $D$ has a unique
 completion $D'$.  
 Let $\pi$ be the permutation of $\delta(w)$ induced by $D'$.
 Let $y\in \delta(w)$ be a vertex of $D$ in the
 short cycle such that $y^\pi=y$.
 Let $x_1,x_2,\ldots,x_m\in \delta(w)$ be the vertices on the long cycle in
 $D'$ such that $x_i^\pi=x_{i+1}$ for all $i=1,2,\ldots,m$ where
 $x_{m+1}=x_1$. Since $D$ has two paths, $m>1$.
 Then $O_P=\{\{x_i,x_{i+1}\}:i=1,2,\ldots,m\}$ 
 and 
 $O_C=\{\{y,x_i\}: i=1,2,\ldots,m\}$ are orbits of $\pi^{(2)}$.
 In $O_P$, each $x_i$ appears twice if $m>2$ and once if $m=2$.
 In $O_C$ each $x_i$ appears once and $y$ appears $m$ times.
 Now it is routine to create a list of orbits satisfying (i) and (ii)
 by taking copies of $O_C$ and copies of $O_P$.

  \medskip\noindent Case 4: Suppose that $D$ is of type (5) having both red and
  blue vertices or $D$ is of type (7) or (8). 
  Let $D'$ be a completion of $D$ obtained by making each path
  of $D$
  to be a cycle of $D'$.
  Let $x_1,x_2,\ldots,x_m\in \delta(w)$ be vertices in the long cycle
  of $D'$ so that $x_i^\pi=x_{i+1}$ for all $i=1,2,\ldots,m$  where
  $x_{m+1}=x_1$.
  Let $y_1,y_2,\ldots,y_k\in \delta(w)$ be vertices in short cycles of $D'$
  such that $y_i^\pi=y_i$.
  Since $D$ is calm, $k\le m$.
  Let 
  $O_j=\{\{x_i,y_j\}: i=1,2,\ldots,m\}$ for $j=1,2,\ldots,k$
  and 
  $O_P=\{\{x_i,x_{i+1}\}: i=1,2,\ldots,m\}$ where
  $x_{m+1}=x_{1}$.
  In the list of $O_1$, $O_2$, $\ldots$, $O_k$, 
  each $x_i$ appears $k$ times 
  and each $y_j$ appears $m$ times.
  In $O_P$, each $x_i$ appears twice if $m>2$ and once if $m=2$.
 To satisfy (i) and (ii), we can take a list of two copies of each $O_j$
  for $j=1,2,\ldots,k$
  and copies of $O_P$.

  \medskip\noindent Case 5: 
  Suppose that $D$ is a directed graph of type (5) not having both red
  and blue, 
  or $D$ is a directed graph of type (6). 
  Then $D$ has a completion $D'$ consisting of a single cycle.
  Let $\pi$ be the permutation of $\delta(w)$ induced by $D'$.
  Let $x_1,x_2,\ldots,x_m\in\delta(w)$ be vertices in $D$ such that 
  $x_i^\pi=x_{i+1}$ for all $i=1,2,\ldots,m$.
 We 
  $O_P=\{\{x_i,x_{i+\lfloor m/2\rfloor}\}: i=1,2,\ldots,m\}$ where
  $x_{j+m}=x_{j}$ for all $j=1,\cdots,\lfloor m/2\rfloor$.
  Then in $O_P$, each $x_i$ appears twice if $m$ is odd and once if
  $m$ is even. 
  Moreover, since all the vertices of the same color appear
  consecutively  in $D'$ and the number of vertices of the same color
  is at most half of $m$, $O_P$ never contains a pair $\{x_i,x_j\}$ of
  vertices of the same color, red or blue.
  Therefore $O_P$ satisfies (i) and (ii).
  This completes the proof.
\end{proof}

\begin{LEM}\label{lem:uniform}
  There exists a completion $D'$ of $D$ so that the $w$-good
  permutation induced by $D'$ is uniform.
\end{LEM}

\begin{proof}
  By Lemma~\ref{lem:oddpath}, $D$ can be partitioned into calm
  subgraphs $D_1$, $D_2$, $\ldots$, $D_t$ of type (1), (2), $\ldots$, (8).
  Lemma~\ref{lem:uniformtype} shows that each $D_i$
  admits a completion that induces a $w$-good uniform permutation
  $\pi_i$
  with a list $L_i$ of orbits of $\pi_{i}^{(2)}$ satisfying (i) and
  (ii).
  Let us assume that each vertex of $D_i$ appears $c_i>0$ times in
  $L_i$.
  Let $c=\operatorname{lcm}(c_1,c_2,\ldots,c_t)$.
  Then let $L$ be the list of orbits obtained by 
  taking   $c/c_i$ copies of $L_i$
  for each $i=1,2,\ldots,t$.
  Then $L$ satisfies (i) and (ii).
  This proves the lemma.
\end{proof}

\begin{REM}
Let us describe combinatorial group theoretic meaning 
of Lemma~\ref{lem:uniform} in the setting of $U\subseteq F_2$
and $G=W(U)$.
For each component of $D$, there exists a cyclic subword, 
called a \emph{$w$-chunk}, of the form
\[
g=u^{\pm1} w^{\pm k}u^{\pm1}
\]
in $U$ such that each length-two subword appearing in $g$
corresponds to one or two vertices (say, $e_i$ or $f_j$) in that component.
Note that we consider only length-two subwords, not cyclic subwords.
For example, the length-five path in Figure~\ref{fig:ex} comes from
the following $w$-chunk in $U$ or $U^{-1}$:
\[u^{-1}w^3u=
\lefteqn{\underbrace{\phantom{u^{-1}\cdot w}}_{f_2}}u^{-1}
\cdot
\lefteqn{\overbrace{\phantom{w\cdot w}}^{e_2=f_3}}w
\cdot 
\lefteqn{\underbrace{\phantom{w\cdot w}}_{e_3=f_4}}w
\cdot
\overbrace{w\cdot u}^{e_4}.\]
The conclusion of Lemma~\ref{lem:uniform} claims that one can partition 
the set of $w$-chunks from $U$ into
$A_1,A_2,\ldots,A_m$ 
such that the ``Whitehead graph'' $G_i$ of each subcollection $A_i$
admits a $w$-good uniform permutation. 
Here, the quotation mark means that
the edges of $G_i\subseteq G$ correspond to length-two subwords from words in $A_i$ and length-two \emph{cyclic} subwords of the form $u^{\pm1}u^{\pm1}$ are not counted.\end{REM}

Now we are ready to prove Conjecture~\ref{con} for $4$-vertex graphs:

\begin{THM}\label{thm:4vertex}
  Let $G$ be the Whitehead graph for a minimal diskbusting list of words in $F_2$. Then $G$ admits a balanced list $\mathcal{C}$ of cycles such that each edge of $G$ appears in the same number of cycles in $\mathcal{C}$.
\end{THM}
\begin{proof}
  Let $w$ be a vertex of minimum degree.
  By Lemma~\ref{lem:uniform}, $G$ has a $w$-good uniform permutation
  $\pi$ on $\delta (w)$.
  By Lemma~\ref{lem:inductive}, 
  $G$ has a nonempty list of cycles satisfying (a), (b), and (c).
\end{proof}
We remark that Conjecture~\ref{con} is true for subdivisions of connected $4$-vertex graphs because of the ``same number of appearances'' condition of Theorem~\ref{thm:4vertex}.
In particular, we verify Tiling Conjecture for rank-two free groups:

\begin{COR}\label{cor:f2}
A minimal and diskbusting list of cyclically reduced words in $F_2$ is polygonal.
\end{COR}

Now Theorem~\ref{thm:main 1} is an immediate consequence of the following.

\begin{COR}\label{cor:f2 2}
For a list $U$ of words in $F_2$, the following are equivalent.
\begin{enumerate}
\item
The list $U$ is diskbusting.
\item
$D(U)$ contains a hyperbolic surface group.
\item
$D(U)$ is one-ended.
\end{enumerate}
\end{COR}

\begin{proof}
(1)$\Leftrightarrow$(3) is well-known and stated in~\cite{GW2010}, for example. 
(1)$\Rightarrow$(2) follows from Corollary~\ref{cor:f2} and Theorem~\ref{thm:polygonal}.
By putting $n=2$ in the Proposition~\ref{prop:tilingequiv}, we have (2)$\Rightarrow$(1).
\end{proof}


\section{Final Remarks}\label{sec:final}
\subsection*{Minimality assumption in Tiling Conjecture}
A graph $G$ is \emph{$2$-connected} if $\abs{V(G)}>2$, 
$G$ is connected, and $G\setminus x$ is connected for every vertex $x$.
It is well-known that a list $U$ of cyclically reduced words in $F_n$ is diskbusting if and only if $W(\phi(U))$ is $2$-connected 
for some $\phi\in\aut(F_n)$~\cite{Stong1997,Stallings1999}. 
However, the minimality assumption in
Tiling Conjecture cannot be weakened to the 
$2$-connectedness of the Whitehead graph;
this is equivalent to saying that the requirement $\lambda(v,v^{-1})=\deg(v)$
in Conjecture~\ref{con} cannot be relaxed to $2$-connectedness.
Daniel Kr\'al'~\cite{Kral2010} kindly provided us Example~\ref{ex:connectivity} showing why this relaxation is not possible. 

\begin{EX}\label{ex:connectivity}
Let $G$ be a $4$-vertex graph  shown in Figure~\ref{fig:connectivity}.
For a vertex $v$ and edges $e\in \delta(v)$ and $f\in \delta(v^{-1})$, we let $f^v=e$ if and only if the number written on $e$ near $v$ coincides with the number written on $f$ near $v^{-1}$.
Actually, $G$ is the Whitehead graph for $w=a (ab^{-1})^3 b^{-2}$.
In Appendix~\ref{sec:appendix}, we will explain why $w$ is not polygonal,
or equivalently,
this graph $G$ does not admit a balanced list of cycles.
Note that $G$ is $2$-connected but $\lambda(a,a^{-1}) = 3 < 4 = \deg(a)$. 
\end{EX}

\begin{figure}
  \centering
  \tikzstyle{v}=[circle,solid, solid, fill=black, inner sep=0pt, minimum width=4pt]
 \tikzstyle{every edge}=[-,draw]
  \tikzstyle{l}=[circle,inner sep=0pt,fill=white]
  \begin{tikzpicture}
    \node [v] at (-3,0) (v0) {};
    \node [v] at (-3,2) (v1) {} 
    edge node [l,pos=0.2] {${}_1$} node [l,pos=0.8] {${}_2$} (v0)
	;
 	\node [v] at (-1,0) (v2) {}
    edge node [l,pos=0.2] {${}_1$} node [l,pos=0.8] {${}_3$} (v0)
    edge [bend right] node [l,pos=0.2] {${}_5$} node [l,pos=0.8] {${}_1$} (v0)
    edge [bend left] node [l,pos=0.2] {${}_2$} node [l,pos=0.8] {${}_4$} (v0)
	;
   \node [v] at (-1,2) (v3) {}
   edge node [l,pos=0.2] {${}_4$} node [l,pos=0.8] {${}_3$} (v2)
    edge [bend left] node [l,pos=0.2] {${}_5$} node [l,pos=0.8] {${}_4$} (v2)
    edge [bend left]   node [l,pos=0.2] {${}_3$} node [l,pos=0.8] {${}_4$}  (v1)
    edge   node [l,pos=0.2] {${}_2$} node [l,pos=0.8] {${}_3$}  (v1)
    edge [bend right] node [l,pos=0.2] {${}_1$} node [l,pos=0.8] {${}_2$} (v1);
    \draw (v0) node [left] {$a^{-1}$};
    \draw (v1) node [left] {$a$};
    \draw (v2) node [right] {$b^{-1}$};
    \draw (v3) node [right] {$b$};
\end{tikzpicture} 
  \caption{Example~\ref{ex:connectivity}.}
  \label{fig:connectivity}
\end{figure}

\subsection*{Non-virtually geometric words}
Let $H_n$ denote a $3$-dimensional handlebody of genus $n$. 
A word $w$ in $F_n$ can be realized as an embedded curve $\gamma\subseteq H_n$.
A word $w$ is said to be \emph{virtually geometric} 
if there exists a finite cover $p\co H'\to H_n$ such that
$p^{-1}(\gamma)$ is homotopic to a $1$-submanifold on the boundary of $H'$~\cite{GW2010}.
Using Dehn's lemma, Gordon and Wilton~\cite{GW2010} proved that if $w\in F_n$ is diskbusting and virtually geometric, then $D(\{w\})$ contains a surface group;
this also follows from the fact that a minimal diskbusting geometric word is polygonal~\cite{Kim2009}.
On the other hand, Manning provided examples of minimal diskbusting, non-virtually
geometric words as follows. 

\begin{THM}[Manning~\cite{Manning2009}]\label{thm:manning}
If the Whitehead graph for a word $w$ in $F_n$ is non-planar, $k$-regular and $k$-edge-connected for some
$k\ge3$, then $w$ is not virtually geometric.
\end{THM}

Here, a graph $G$ is said to be \emph{$k$-edge-connected} if $\abs{\delta(X)} \ge k$ for all $\emptyset\neq X\subsetneq V(G)$. So, if $W(U)$ is $k$-regular and $k$-edge-connected for a list $U$ of words in $F_n$, then $U$ is minimal and diskbusting.
Hence even for the words provided by Manning, Theorem~\ref{thm:main 2} finds hyperbolic surface groups in the corresponding doubles:

\begin{COR}\label{cor:regular}
If the Whitehead graph for a list $U$ of words in $F_n$ is $k$-regular and $k$-edge-connected for some
$k\ge3$, then $U$ is polygonal. In particular, $D(U)$ contains a hyperbolic surface group.
\end{COR}

\bibliographystyle{amsplain}
\def\soft#1{\leavevmode\setbox0=\hbox{h}\dimen7=\ht0\advance \dimen7
  by-1ex\relax\if t#1\relax\rlap{\raise.6\dimen7
  \hbox{\kern.3ex\char'47}}#1\relax\else\if T#1\relax
  \rlap{\raise.5\dimen7\hbox{\kern1.3ex\char'47}}#1\relax \else\if
  d#1\relax\rlap{\raise.5\dimen7\hbox{\kern.9ex \char'47}}#1\relax\else\if
  D#1\relax\rlap{\raise.5\dimen7 \hbox{\kern1.4ex\char'47}}#1\relax\else\if
  l#1\relax \rlap{\raise.5\dimen7\hbox{\kern.4ex\char'47}}#1\relax \else\if
  L#1\relax\rlap{\raise.5\dimen7\hbox{\kern.7ex
  \char'47}}#1\relax\else\message{accent \string\soft \space #1 not
  defined!}#1\relax\fi\fi\fi\fi\fi\fi}
\providecommand{\bysame}{\leavevmode\hbox to3em{\hrulefill}\thinspace}
\providecommand{\MR}{\relax\ifhmode\unskip\space\fi MR }
\providecommand{\MRhref}[2]{%
  \href{http://www.ams.org/mathscinet-getitem?mr=#1}{#2}
}
\providecommand{\href}[2]{#2}

\appendix
\section{Example~\ref{ex:connectivity}}\label{sec:appendix}
We will show that the word $w$ in Example~\ref{ex:connectivity} is not
polygonal. We do this by reducing the question to a nonnegative integral LP problem which is noted in Remark~\ref{rem:polynomial time}.
This reduction easily generalizes to free groups of arbitrary ranks.

We follow the notation in Example~\ref{ex:connectivity}.
Let $\mathcal{C}$ be the set of simple cycles in $G$, and 
 $\mathcal{R}$ be the disjoint union of $\{(a,i,j) : 1\le i<j\le 4\}$ and $\{(b,i,j) : 1\le i<j\le 5\}$.
For each $C\in \mathcal{C}$, we define $p_C$ to be the $ \mathcal{R}\times 1$ vector such that:
\begin{enumerate}[(i)]
\item
the $(a,i,j)$-th row element of $p_C$ is $1$ if $C$ contains the two edges labeled by $i$ and $j$ incident with the vertex $a$, and $0$ otherwise;
\item
the $(b,i,j)$-th row element of $p_C$ is $1$ if $C$ contains the two edges labeled by $i$ and $j$ incident with the vertex $b$, and $0$ otherwise.
\end{enumerate}
Similarly, an  $ \mathcal{R}\times 1$ vector $q_C$ is defined by:
\begin{enumerate}[(i)]
\item
the $(a,i,j)$-th row element of $q_C$ is $1$ if $C$ contains the two edges labeled by $i$ and $j$ incident with the vertex $a^{-1}$, and $0$ otherwise;
\item
the $(b,i,j)$-th row element of $q_C$ is $1$ if $C$ contains the two edges labeled by $i$ and $j$ incident with the vertex $b^{-1}$, and $0$ otherwise.
\end{enumerate}
Now we define an $\mathcal{R}\times \mathcal{C}$ matrix $M$ by
declaring that the $C$-column is the vector $p_C-q_C$ for each
$C\in\mathcal{C}$. Then $w$ is polygonal if and only if there exists a non-zero nonnegative integer solution to $Mx = 0$.
With suitable orderings on $\mathcal{C}$ and $\mathcal{R}$,
we obtain the matrix $M$ in Figure~\ref{fig:m}.
\begin{figure}
\setcounter{MaxMatrixCols}{30}
\[\left(
 \begin{smallmatrix}
  0 & 0 & 0 & 0 & 0 & -1 & 0 & 1 & 1 & 1 & 1 & 1 &  1 & 0 & 0 & 0 & 0 & 0 & 0 & 0 & 0 & 0 & 0 & 0 & 0\\
  0 & 0 & 0 & 0 & -1 & 0 & 0 & 0 & 0 & 0 & 0 & 0 & 0 & 1 & 1 & 1 & 1 & 1 & 1 & 0 & 0 & 0 & 0 & 0 & 0 \\     
  0 & 1 & 0 & 0 & 0 & 0 & -1 & 0 & 0 & 0 & 0 & 0 & 0 & 0 & 0 & 0 & 0 & 0 & 0 & 0 & 0 & 0 & 0 & 0 & 0 \\     
  0 & 0 & 0 & 0 & 0 & 0 & 0 & -1 & 0 & 0 & -1 & 0 & 0 & -1 & 0 & 0 & -1 & 0 & 0 & 0 & 1 & 1 & 0 & 1 & 1 \\
  0 & 0 & 1 & 0 & 0 & 0 & 0 & 0 & 0 & -1 & 0 & 0 & -1 & 0 & 0 & -1 & 0 & 0 & -1 & 0 & 0 & -1 & 0 & 0 & -1 \\
  0 & 0 & 0 & 1 & 0 & 0 & 0 & 0 & -1 & 0 & 0 & -1& 0 & 0 & -1 & 0 & 0 & -1 & 0 & 0 & -1 & 0 & 0 & -1 & 0 \\
  1 & 0 & 0 & 0 & 0 & 0 & 0 & -1 & 0 & 0 & 0 & 0 & 0 & -1 & 0 & 0 & 0 & 0 & 0 & -1 & 0 & 0 & 0 & 0 & 0 \\   
  0 & 0 & 0 & 0 & 0 & 0 & 0 & 1 & 1 & 1 & -1 & 0 & 0 & 0 & 0 & 0 & -1 & 0 & 0 & 0 & 0 & 0 & -1 & 0 & 0 \\   
  -1 & 0 & 0 & 0 & 0 & 0 & 0 & 0 & 0 & 0 & 1 & 1 & 1 & 0 & 0 & 0 & 0 & 0 & 0 & 0 & 0 & 0 & 0 & 0 & 0 \\     
  0 & 0 & 0 & 0 & 0 & -1 & 0 & 0 & 0 & 0 & 0 & 0 & 0 & 1 & 1 & 1 & 0 & 0 & 0 & 0 & 0 & 0 & 0 & 0 & 0 \\     
  0 & 0 & 0 & 0 & 0 & 0 & 0 & 0 & 0 & -1 & 0 & 0 & 0 & 0 & 0 & -1 & 1 & 1 & 1 & 0 & 0 & -1 & 0 & 0 & 0 \\   
  0 & 1 & 0 & 0 & 0 & 0 & 0 & 0 & 0 & 0 & 0 & 0 &  -1 & 0 & 0 & 0 & 0 & 0 & -1 & 0 & 0 & 0 & 0 & 0 & -1 \\  
  0 & 0 & 0 & 0 & -1 & 0 & 0 & 0 & 0 & 0 & 0 & 0 & 0 & 0 & 0 & 0 & 0 & 0 & 0 & 1 & 1 & 1 & 0 & 0 & 0 \\     
  0 & 0 & 0 & 0 & 0 & 0 & 0 & 0 & -1 & 0 & 0 & 0 & 0 & 0 & -1 & 0 & 0 & 0 & 0 & 0 & -1 & 0 & 1 & 1 & 1 \\   
  0 & 0 & 1 & 0 & 0 & 0 & 0 & 0 & 0 & 0 & 0 & -1 & 0 & 0 & 0 & 0 & 0 & -1 & 0 & 0 & 0 & 0 & 0 & -1 & 0 \\   
  0 & 0 & 0 & 1 & 0 & 0 & -1 & 0 & 0 & 0 & 0 & 0 & 0 & 0 & 0 & 0 & 0 & 0 & 0 & 0 & 0 & 0 & 0 & 0 & 0
 \end{smallmatrix}\right)
\]
  \caption{Matrix $M$}
  \label{fig:m}
\end{figure}
Using computer algebra softwares such as Mathematica, one can verify that $Mx=0$ does not have non-zero nonnegative integer solutions. Hence, $w$ is not polygonal.

\end{document}